\documentclass[11pt]{amsart}
\usepackage[latin1]{inputenc}
\usepackage{subfigure}
\usepackage{color}
\usepackage{amsmath}
\usepackage{amsthm}
\usepackage{amstext}
\usepackage{amssymb}
\usepackage{amsfonts}
\usepackage{graphicx}
\usepackage{young}
\usepackage{multicol}
\usepackage{mathrsfs}
\usepackage{stmaryrd}
\usepackage{bbm}
\usepackage[all]{xy}
\usepackage{mathabx}
\usepackage{tikz}
\usepackage{mathtools}
\usepackage{tabularx}
\usepackage{array}
\usepackage{commath}
\usepackage{ytableau}
\usepackage{caption}

\theoremstyle{plain}
\theoremstyle{definition}
\newtheorem{theorem}{Theorem}[section]
\newtheorem{conjecture}[theorem]{Conjecture}
\newtheorem{remark}[theorem]{Remark}
\newtheorem{lemma}[theorem]{Lemma}
\newtheorem{definition}[theorem]{Definition}

\newtheorem{example}[theorem]{Example}
\newtheorem{proposition}[theorem]{Proposition}
\newtheorem{corollary}[theorem]{Corollary}

\newtheorem*{theorem*}{Theorem}

\usepackage{amssymb,amsmath,tabularx,graphicx}
\usepackage{tikz}
\usetikzlibrary[calc,intersections,through,backgrounds,arrows,decorations.pathmorphing]

\begin{document}
\date{}
\title{Promotion on Generalized Oscillating Tableaux and Web Rotation}
\author{Rebecca Patrias}
\address{Laboratoire de Combinatoire et d'Informatique Math\'ematique,
Universit\'e du Qu\'ebec \`a Montr\'eal}
\email{patriasr@lacim.ca}

\begin{abstract}
We introduce the notion of a generalized oscillating tableau and define a promotion operation on such tableaux that generalizes the classical promotion operation on standard Young tableaux. As our main application, we show that this promotion corresponds to rotation of the irreducible $A_2$-webs of G. Kuperberg. 
\end{abstract}

\maketitle

\section{Introduction}

Recall that a \textit{partition} is a finite, nonincreasing sequence of positive integers $\lambda=(\lambda_1,\ldots,\lambda_t)$ and that any partition can be identified with the corresponding \textit{Young diagram}---a left-justified array of boxes with $\lambda_i$ boxes in the $i$th row from the top. An \textit{oscillating tableau} of length $k$ is a sequence of $k+1$ partitions $(\lambda^0=\emptyset,\ldots,\lambda^k)$, where $\lambda^0=\emptyset$ and $\lambda^i$ is obtained from $\lambda^{i-1}$ by either adding or deleting one box. In this paper, we generalize these notions. 

We define a \textit{generalized partition with $n$ parts} $\lambda=(\lambda_1\geq\cdots\geq\lambda_n)$ to be a nonincreasing list of $n$ (not necessarily positive) integers. We introduce the notion of a \textit{generalized oscillating tableau of length $k$ with $n$ parts}: a sequence of $k+1$ generalized partitions $(\emptyset,\lambda^1,\ldots,\lambda^k)$ such that each $\lambda^i$ has $n$ parts, $\lambda^0=\emptyset=(0,\ldots,0)$, and $\lambda^{i+1}$ can be obtained from $\lambda^i$ by either adding or subtracting 1 from one of $\lambda^i_1,\ldots,\lambda^i_n$. We visualize generalized partitions using a generalization of Young diagrams, where we allow negative row sizes and indicate negative rows by coloring the corresponding boxes red. We may then associate a set-valued tableau $T$ to each generalized oscillating tableau, where the set of boxes of $T$ is the union of boxes in $\lambda^1,\ldots,\lambda^k$ and we add entry $i$ (resp. $i'$) to the subset of primed and unprimed positive integers in a box if $\lambda^i$ is obtained from $\lambda^{i+1}$ by adding (resp. deleting) the corresponding box. For example, the generalized oscillating tableau of length 5 with 2 parts 
\[((0,0), (1,0,), (1,-1), (2,-1),(2,0), (1,0))\] corresponds to the set-valued filling below.
\begin{center}
\begin{ytableau}
\none & 1 & 35'\\
*(red) 2'4
\end{ytableau}
\end{center}

Let $\text{GOT}(k,n)$ denote the set of generalized oscillating tableaux of length $k$ with $n$ parts. We define a promotion operation $p:\text{GOT}(k,n)\to\text{GOT}(k,n)$ that generalizes classical tableau promotion. We define this promotion operation using both growth rules and growth diagrams and using tableau rules. Figure~\ref{fig:genoscprom} shows an example of generalized oscillating promotion. A reader familiar with promotion on standard Young tableaux will recognize the similarities.  

\begin{figure}[h!]
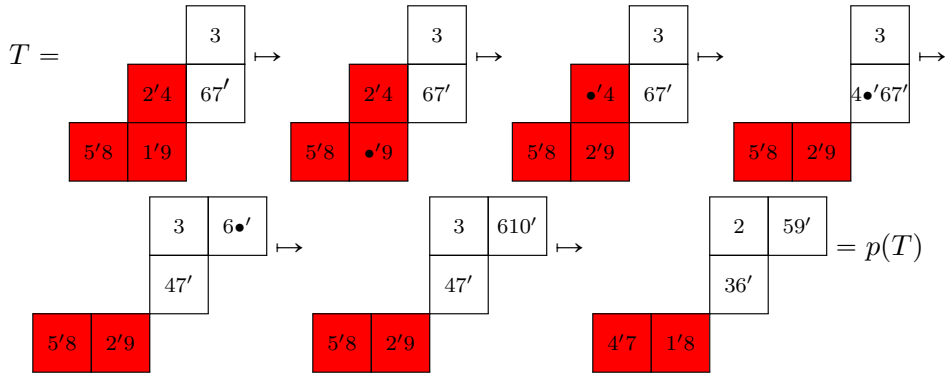

\[
\ytableausetup{boxsize=.3in}
T=
\begin{ytableau}
\none & \none & \scriptstyle{3} \\
\none & *(red) {\scriptstyle{2'4}} & {\scriptstyle{67}}' \\
*(red) {\scriptstyle{5'8}}  &*(red) {\scriptstyle{1'9}}
\end{ytableau}\mapsto
\begin{ytableau}
\none & \none & {\scriptstyle{3}} \\
\none & *(red) {\scriptstyle{2'4}} & {\scriptstyle{67'}} \\
*(red) {\scriptstyle{5'8}}  &*(red) {\scriptstyle{\bullet'9}}
\end{ytableau}\mapsto
\begin{ytableau}
\none & \none & {\scriptstyle{3}} \\
\none & *(red) {\scriptstyle{\bullet'4}} & {\scriptstyle{67'}} \\
*(red) {\scriptstyle{5'8}}  &*(red) {\scriptstyle{2'9}}
\end{ytableau}\mapsto
\begin{ytableau}
\none & \none & {\scriptstyle{3}} \\
\none & \none & {\scriptstyle{4\bullet'67'}} \\
*(red) {\scriptstyle{5'8}}  &*(red) {\scriptstyle{2'9}}
\end{ytableau}\mapsto\]
\[
\begin{ytableau}
\none & \none & {\scriptstyle{3}} & {\scriptstyle{6 \bullet'}} \\
\none & \none & {\scriptstyle{47'}} \\
*(red) {\scriptstyle{5'8}}  &*(red) {\scriptstyle{2'9}}
\end{ytableau}\mapsto
\begin{ytableau}
\none & \none & {\scriptstyle{3}} & {\scriptstyle{6 10'}} \\
\none & \none & {\scriptstyle{47'}} \\
*(red) {\scriptstyle{5'8}}  &*(red) {\scriptstyle{2'9}}
\end{ytableau}\mapsto
\begin{ytableau}
\none & \none & {\scriptstyle{2}} & {\scriptstyle{5 9'}} \\
\none & \none & {\scriptstyle{36'}} \\
*(red) {\scriptstyle{4'7}}  &*(red) {\scriptstyle{1'8}}
\end{ytableau}
=p(T)
\]
\caption{We start with generalized oscillating tableau $T$ and construct its image under generalized oscillating promotion, $p(T)$.}\label{fig:genoscprom}
\end{figure}

As our main application, we relate generalized oscillating promotion on $\text{GOT}(k,3)$ to rotation of irreducible $A_2$-webs. An \textit{irreducible $A_2$-web} can be defined as a bipartite graph with fixed coloring embedded in a disk such that each vertex on the boundary of the disk has degree 1, each interior vertex has degree 3, and all internal faces have at least 6 sides. Webs were defined by G. Kuperberg motivated by the study of multilinear invariant theory \cite{kuperberg1996spiders}. In his paper, Kuperberg introduces combinatorial rank 2 spiders, which are a diagrammatic presentation of the space Inv$(V_1 \otimes \cdots \otimes V_n)$, i.e., the invariant space of a tensor product of irreducible representations $V_i$ of a rank 2 Lie algebra $\mathfrak{g}$. Webs are a basis for the invariant space in this diagrammatic presentation.

Webs have since been studied by G. Kuperberg and M. Khovanov \cite{khovanov1999web}; T.K. Peterson, P. Pylyavskyy, and B. Rhoades \cite{petersen2009promotion}; S. Fomin and P. Pylyavskyy \cite{fomin2012tensor}; and many others. In particular, Khovanov and Kuperberg describe a bijection between webs and \textit{signature and state strings}: a vector of pairs, where each pair $(j_i,s_i)\in \{\bullet,\circ\}\times\{1,0,\bar{1}\}$. Using this correspondence between webs and signature and state strings, it is easy to associate to each web with all black boundary vertices a three-row standard Young tableau of rectangular shape. In their paper, Peterson, Pylyavskyy, and Rhoades describe how to interpret the action of tableau promotion on these rectangular tableaux as web rotation. We generalize this result in this paper.

Using the signature and state strings of Khovanov and Kuperberg \cite{khovanov1999web}, we associate to each web with fixed first/leftmost vertex and $k$ boundary vertices a generalized oscillating tableau of length $k$ with $3$ parts. Note that in contrast to the work of Peterson--Pylyavskyy--Rhoades, we do not restrict to webs whose boundary vertices are all the same color. Our main result is the following.

\begin{theorem}\label{thm:mainthm}
Let $D$ be a web with fixed leftmost vertex. The generalized oscillating tableau associated with counterclockwise rotation of $D$ is given by generalized oscillating promotion of the tableau associated with $D$ itself.
\end{theorem}

The paper proceeds as follows. In Section~\ref{sec:preliminaries}, we review basic definitions related to tableaux and promotion on tableaux. We introduce irreducible $A_2$-webs and describe the correspondence between webs and signature and state string. We end the section by discussing the bijection between three-row, rectangular standard Young tableaux and webs with all black boundary vertices. In Section~\ref{sec:blackpromotionandrotation}, we summarize the results of Peterson--Pylyavksyy--Rhoades on promotion and web rotation for webs with all black boundary vertices. Section~\ref{sec:osctablandpromotion} introduces generalized oscillating tableaux and generalized oscillating promotion as well as explains how to associate a generalized oscillating tableau to a web. We prove Theorem~\ref{thm:mainthm} in Section~\ref{sec:rotationandgenoscprom}. We end by stating some related open questions in Section~\ref{sec:future}. 

\section{Preliminaries}\label{sec:preliminaries}
\subsection{Tableaux and Promotion}\label{sec:tabandprom}
Recall that a \textit{partition} is a finite, weakly decreasing sequence of positive integers $\lambda=(\lambda_1\geq \lambda_2\geq\cdots\geq \lambda_k>0)$. To any partition we may associate a \textit{Young diagram}---a left-justified array of boxes (or cells) with rows that weakly descend in size from top to bottom---by having $\lambda_i$ boxes in the $i$th row. We equate these two notions and refer to both a partition and its Young diagram as $\lambda$. Let $|\lambda|=\lambda_1+\cdots+\lambda_k$ denote the number of boxes in $\lambda$. A \textit{standard Young tableau} of shape $\lambda$ is a filling of the cells of a Young diagram of shape $\lambda$ with $1,2,\ldots,|\lambda|$ such that entries in rows and columns are increasing and each entry appears exactly once. We say partition $\mu$ is contained in $\lambda$ if the Young diagram for $\mu$ fits inside that for $\lambda$, and in this case, we let $\lambda/\mu$ denote the set of boxes of $\lambda$ that are not also in $\mu$.

We next describe an action on standard Young tableaux of shape $\lambda$ called \textit{jeu de taquin promotion}, or simply \textit{promotion}. Promotion was originally defined as an action on partially ordered sets by Sch\"utzenberger \cite{schutzenberger1972promotion}.

\begin{definition}[Promotion]
Given a standard Young tableau $T$ of shape $\lambda$ with $|\lambda|=k$, form $p(T)$ using the following steps.
\begin{enumerate}
\item Delete the entry 1 from the box in the upper lefthand corner of $T$ and replace it with $\bullet$.
\item Perform jeu de taquin: For each $i\in\{2,\ldots,k\}$, perform the following swap with $i$ starting at 2 and consecutively increasing after each swap.
\begin{enumerate}
\item If the box containing $i$ is directly below or directly right of the box containing $\bullet$, switch the labels of the two boxes.
\item If the box containing $i$ is not directly below and not directly right of the box containing $\bullet$, do nothing.
\end{enumerate}
\item Delete the $\bullet$ and fill its box with $k+1$.
\item Subtract 1 from each entry.
\end{enumerate}
\end{definition}

\begin{example}\label{ex:promotion}
Starting with standard Young tableau $T$, we obtain $p(T)$ using promotion.
\ytableausetup{boxsize=.2in}
\begin{eqnarray*}
T=
\begin{ytableau}
1 & 2 & 6 \\
3 & 5 & 7\\
4 & 8 & 9
\end{ytableau}\hspace{.1in} \mapsto \hspace{.1in}
\begin{ytableau}
\bullet & 2 & 6 \\
3 & 5 & 7\\
4 & 8 & 9
\end{ytableau} &\mapsto& 
\begin{ytableau}
2 & \bullet & 6 \\
3 & 5 & 7\\
4 & 8 & 9
\end{ytableau}\hspace{.1in} \mapsto \hspace{.1in}
\begin{ytableau}
2 & 5 & 6 \\
3 & \bullet & 7\\
4 & 8 & 9
\end{ytableau}\hspace{.1in} \mapsto \\
\begin{ytableau}
2 & 5 & 6 \\
3 & 7 & \bullet\\
4 & 8 & 9
\end{ytableau}\hspace{.1in} \mapsto \hspace{.1in}
\begin{ytableau}
2 & 5 & 6 \\
3 & 7 & 9\\
4 & 8 & \bullet
\end{ytableau} &\mapsto& 
\begin{ytableau}
2 & 5 & 6 \\
3 & 7 & 9\\
4 & 8 & 10
\end{ytableau}\hspace{.1in} \mapsto \hspace{.1in}
\begin{ytableau}
1 & 4 & 5 \\
2 & 6 & 8\\
3 & 7 & 9
\end{ytableau}
=p(T)
\end{eqnarray*}
\end{example}

We can equivalently describe promotion using \textit{promotion growth diagrams} and a set of \textit{promotion growth rules}, as we now explain. We refer the reader to \cite{stanley1999enumerative} for more details. Suppose we wish to perform promotion on standard Young tableau $T$ with $k$ boxes. First, write $T$ as a sequence of partition shapes $(\lambda^0=\emptyset, \lambda^1,\ldots,\lambda^k)$ starting with the empty shape, where $\lambda^i$ is obtained from $\lambda^{i-1}$ by adding the box with label $i$ in $T$. We inductively create a new sequence $(\mu^0=\emptyset, \mu^1,\ldots,\mu^k)=p(T)$ using the following rules.

\underline{Promotion Growth Rules}

Suppose $\lambda^s/\mu^{s-1}$ is a box in row $i$ and $\lambda^{s+1}/\lambda^s$ is a box in row $j$. 
\begin{itemize}
\item If the result of adding a box to $\mu^{s-1}$ in row $j$ is a partition, then $\mu^s$ is the result of adding this box to $\mu^{s-1}$.
\item If the result of adding a box to $\mu^{s-1}$ in row $j$ is not a partition then $\mu^s=\lambda^s$.
\item $\mu^k=\lambda^k$.
\end{itemize}

Starting with a tableau $T$, we illustrate these rules in a growth diagram, as shown in Example~\ref{ex:promgrowthrules}.

\begin{example}\label{ex:promgrowthrules} Let $T$ be the tableau from Example~\ref{ex:promotion}. Then $T$ is represented by the sequence of partitions on the top line, and we construct $p(T)$ on the line below using the growth rules. The image shown is the corresponding growth diagram.

\begin{center}
\begin{tikzpicture}[scale=1.6]
\ytableausetup{boxsize=.1in}
\node (0) at (0,0) {$\emptyset$};
\node (1) at (1,0) {$\ydiagram{1}$};
\node (2) at (2,0) {$\ydiagram{2}$};
\node (3) at (3,0) {$\ydiagram{2,1}$};
\node (4) at (4,0) {$\ydiagram{2,1,1}$};
\node (5) at (5,0) {$\ydiagram{2,2,1}$};
\node (6) at (6,0) {$\ydiagram{3,2,1}$};
\node (7) at (7,0) {$\ydiagram{3,3,1}$};
\node (8) at (8,0) {$\ydiagram{3,3,2}$};
\node (9) at (9,0) {$\ydiagram{3,3,3}$};
\node (0') at (1,-1) {$\emptyset$};
\node (1') at (2,-1) {$\ydiagram{1}$};
\node (2') at (3,-1) {$\ydiagram{1,1}$};
\node (3') at (4,-1) {$\ydiagram{1,1,1}$};
\node (4') at (5,-1) {$\ydiagram{2,1,1}$};
\node (5') at (6,-1) {$\ydiagram{3,1,1}$};
\node (6') at (7,-1) {$\ydiagram{3,2,1}$};
\node (7') at (8,-1) {$\ydiagram{3,2,2}$};
\node (8') at (9,-1) {$\ydiagram{3,3,2}$};
\node (9') at (10,-1) {$\ydiagram{3,3,3}$};
\draw (0)--(1)--(2)--(3)--(4)--(5)--(6)--(7)--(8)--(9);
\draw (0')--(1')--(2')--(3')--(4')--(5')--(6')--(7')--(8')--(9');
\draw (1)--(0') (2)--(1') (3)--(2') (4)--(3') (5)--(4') (6)--(5') (7)--(6') (8)--(7') (9)--(8') ;
\end{tikzpicture}
\end{center}
\end{example}

We can translate between the growth diagram and the tableau description of promotion as follows. Starting with the growth diagram, we of course recover the original standard Young tableau $T$ by reading across the top line. We then create the result of swapping $\bullet$ with $i$ by reading from the growth diagram $(\mu^0, \mu^1,\ldots,\mu^{i-1},\lambda^i,\ldots,\lambda^k)$. When we form the tableau with $\bullet$, we put entry $j+1$ in $\mu^{j}/\mu^{j-1}$, entry $j$ in $\lambda^j/\lambda^{j-1}$, and we put $\bullet$ in the box $\lambda^i/\mu^{i-1}$. For example, reading $(\mu^0,\mu^1,\mu^2,\mu^3,\mu^4,\lambda^5,\lambda^6,\lambda^7,\lambda^8,\lambda^9)$ in Example~\ref{ex:promgrowthrules}, we build
\begin{center}
\ytableausetup{boxsize=.2in}
\begin{ytableau}
2 & 5 & 6 \\
3 & \bullet & 7 \\
4 & 8 & 9
\end{ytableau}
\end{center}  
from Example~\ref{ex:promotion}. This translation will be important in later sections.

\subsection{Webs}
Webs were introduced by G. Kuperberg in the following way.

\begin{definition}[\cite{kuperberg1996spiders}]
An \textit{$A_2$-web} is a planar, directed graph $D$ with no multiple edges embedded in a disk satisfying the following conditions:
\begin{enumerate}
\item $D$ is bipartite, (i.e., each vertex has either all adjacent edges pointing away from it or all adjacent edges pointing toward it)
\item all of the boundary vertices have degree 1, and
\item all internal vertices have degree 3.
\end{enumerate}
An $A_2$-web is \textit{non-elliptic} if
\begin{enumerate}
\item[(4)] all internal faces of $D$ have at least 6 sides.
\end{enumerate}
When all four conditions are satisfied, we call $D$ an \textit{irreducible $A_2$-web.}
\end{definition}

In this document, we will refer to irreducible $A_2$-webs simply as \textit{webs}. In other words, all webs are assumed to be irreducible $A_2$-webs. We will also omit the directions of the edges of a web $D$ and instead bicolor the vertices of $D$. A vertex $v$ will be black if all adjacent edges point toward $v$ and will be white if all adjacent edges point away from $v$. We view webs as combinatorial objects and thus are only concerned with webs up to homeomorphism on the interior of the disk and place boundary vertices canonically. See Figure~\ref{fig:webs} for examples.

\begin{figure}[h!]
\begin{center}
\includegraphics[width=4in]{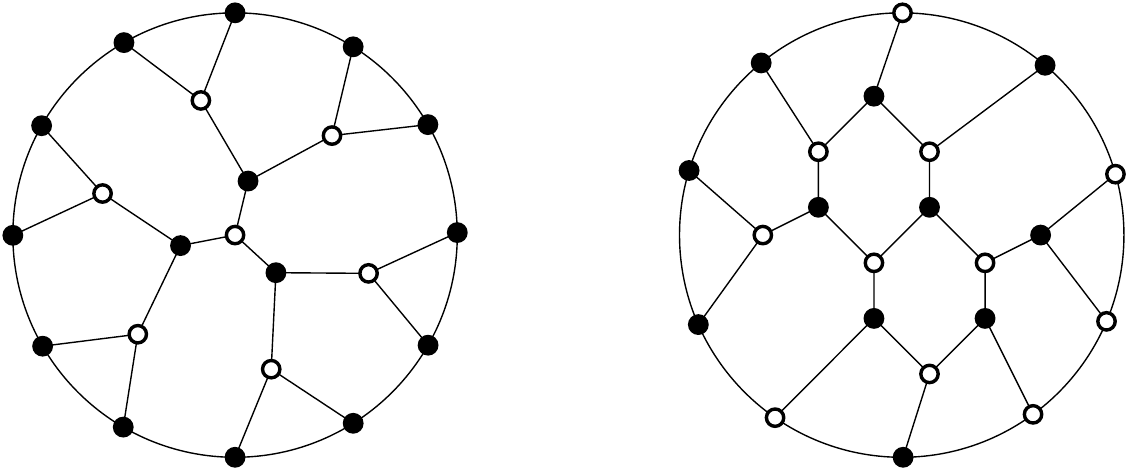}
\end{center}
\caption{Examples of webs. The web on the left is in $\mathcal{M}_4$.}
\label{fig:webs}
\end{figure}

\subsection{Dominant Paths}
In \cite{khovanov1999web}, M. Khovanov and G. Kuperberg describe a bijection between webs with $n$ boundary vertices and fixed leftmost vertex and certain length $n$ strings, which we now describe. 

A \textit{signature} of length $n$ is a sequence $S = (s_1, s_2, \ldots, s_n) \in \{\circ,\bullet\}^n$. A \textit{state string} is a sequence $J = (j_1, j_2, \ldots, j_n) \in \{\bar{1},0,1\}^n$. A \textit{signature and state string} is a sequence $((j_1,s_1), (j_2,s_2),\ldots,(j_n,s_n))$, where each element is a state paired with either $\circ$ or $\bullet$. Khovanov and Kuperberg refer to this string as the \textit{sign and state string}, but we prefer to use the signature terminology as inspired by \cite{fomin2012tensor} since we represent webs with a bicoloring instead of as directed graphs.

Certain signature and state strings correspond to weight lattice paths in a Weyl chamber of sl$(3)$ as follows. Each vector in the signature and state string $(j_k,s_k)$ has a weight $\mu_k$. The imagine in Figure~\ref{fig:weightlattice} shows our convention for the weight for $(1,\bullet)$ in red parallel to the $x$-axis, the weight for $(0,\bullet)$ in blue 120$^\circ$ counterclockwise from the previous, and the weight for $(\bar{1},\bullet)$ in green. The weight for $(j_k,\circ)$ is the negative of the weight for $(j_k,\bullet)$. The \textit{dominant Weyl chamber} is defined as the subset of the weight lattice consisting of positive integral linear combinations of the weights for $(1,\bullet)$ and $(\bar{1},\circ)$ and is shaded gray in the figure. In future sections, we will refer to the extreme ray of the dominant chamber parallel to the weight for $(1,\bullet)$ as the \textit{lower extreme ray} and the extreme ray parallel to the weight for $(\bar{1},\bullet)$ as the \textit{upper extreme ray}.

\begin{figure}[h!]
\begin{center}
\begin{tikzpicture}
\draw[lightgray, top color=lightgray,bottom color=lightgray] (0,0) -- (2,0) -- (2,3.464) -- cycle;
\draw (-2,0)--(2,0);
\draw (0,2)--(0,-2);
\draw[->, very thick, red] (0,0)--(1,0);
\draw[->, very thick, blue] (0,0)--(-.5,.866);
\draw[->, very thick, green] (0,0)--(-.5,-.866);
\end{tikzpicture}
\end{center}
\caption{The weight lattice for sl$(3)$ with dominant chamber shaded gray.}
\label{fig:weightlattice}
\end{figure}
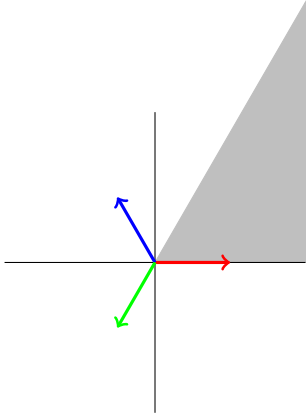

Given a signature and state string with $n$ components, we define a path in the weight lattice of sl$(3)$ $\pi=\pi_0,\ldots,\pi_n$, where $\pi_0$ is the origin and $\pi_k=\pi_{k-1}+\mu_k$. We call such a path \textit{dominant} if it begins and ends at the origin and is contained in the dominant chamber. We call a signature and state string \textit{dominant} if its corresponding path is dominant. For example, the path in Figure~\ref{fig:growthfigure} is easily seen to be dominant, and thus the corresponding signature and state string is dominant.

\subsection{Bijection between strings and webs}
Given a dominant signature and state string,  Khovanov--Kuperberg builds a web by giving a series of inductive rules. To construct a web from a signature and state string with $n$ components, place $n$ vertices on a line segment with colors corresponding to the state string. Draw an edge stemming from each vertex, and label each edge with the corresponding state given in the signature and state string. Next construct the web by following the growth rules shown in Figure~\ref{fig:growthrules}. We emphasize that the $1,0,\bar{1}$ labels in the figure are edge labels. Lastly, glue the ends of the line segment together so the web is contained in a disk. Figure~\ref{fig:growthfigure} shows an example of the web growth rules applied to the signature and state string  
\[\left((1,\bullet),(1,\bullet),(\bar{1},\circ), (0,\bullet), (\bar{1},\circ), (0,\bullet), (\bar{1},\bullet), (\bar{1},\bullet),(0,\circ), (1,\circ),(\bar{1},\bullet)\right).\]

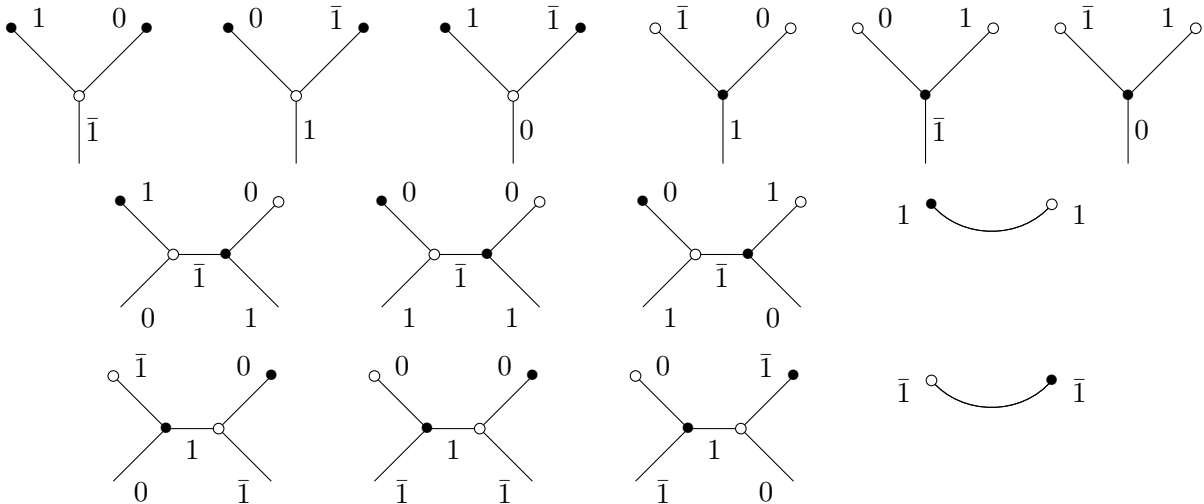
\begin{figure}[h!]
\begin{center}
\begin{tikzpicture}[scale=.9] 
\node (a) at (0,0) {$\bullet$};
\node (b) at (2,0) {$\bullet$};
\node (d) at (1.2,-1.5) {$\bar{1}$};
\draw[-] (1,-1) -- (0,0)  node [above, label=right: 1] {};
\draw[-] (1,-1) -- (1,-2);
\draw[-] (1,-1) -- (2,0) node [above, label=left:0] {};
\node[circle,draw=black, fill=white, inner sep=0pt,minimum size=4pt] (d) at (1,-1) {};
\end{tikzpicture}\hspace{.2in}
\begin{tikzpicture}[scale=.9] 
\node (a) at (0,0) {$\bullet$};
\node (b) at (2,0) {$\bullet$};
\node (d) at (1.2,-1.5) {$1$};
\draw[-] (1,-1) -- (0,0)  node [above, label=right: 0] {};
\draw[-] (1,-1) -- (1,-2);
\draw[-] (1,-1) -- (2,0) node [above, label=left:$\bar{1}$] {};
\node[circle,draw=black, fill=white, inner sep=0pt,minimum size=4pt] (d) at (1,-1) {};
\end{tikzpicture}\hspace{.2in}
\begin{tikzpicture}[scale=.9] 
\node (a) at (0,0) {$\bullet$};
\node (b) at (2,0) {$\bullet$};
\node (d) at (1.2,-1.5) {$0$};
\draw[-] (1,-1) -- (0,0)  node [above, label=right: 1] {};
\draw[-] (1,-1) -- (1,-2);
\draw[-] (1,-1) -- (2,0) node [above, label=left:$\bar{1}$] {};
\node[circle,draw=black, fill=white, inner sep=0pt,minimum size=4pt] (d) at (1,-1) {};
\end{tikzpicture}\hspace{.2in}
\begin{tikzpicture}[scale=.9] 
\node (d) at (1,-1) {$\bullet$};
\draw[-] (1,-1) -- (0,0)  node [above, label=right: $\bar{1}$] {};
\draw[-] (1,-1) -- (1,-2);
\draw[-] (1,-1) -- (2,0) node [above, label=left:0] {};
\node[circle,draw=black, fill=white, inner sep=0pt,minimum size=4pt] (a) at (0,0) {};
\node[circle,draw=black, fill=white, inner sep=0pt,minimum size=4pt] (b) at (2,0) {};
\node (d) at (1.2,-1.5) {1};
\end{tikzpicture}\hspace{.2in}
\begin{tikzpicture}[scale=.9] 
\node (d) at (1,-1) {$\bullet$};
\draw[-] (1,-1) -- (0,0)  node [above, label=right: 0] {};
\draw[-] (1,-1) -- (1,-2);
\draw[-] (1,-1) -- (2,0) node [above, label=left:1] {};
\node[circle,draw=black, fill=white, inner sep=0pt,minimum size=4pt] (a) at (0,0) {};
\node[circle,draw=black, fill=white, inner sep=0pt,minimum size=4pt] (b) at (2,0) {};
\node (d) at (1.2,-1.5) {$\bar{1}$};
\end{tikzpicture}\hspace{.2in}
\begin{tikzpicture}[scale=.9] 
\node (d) at (1,-1) {$\bullet$};
\draw[-] (1,-1) -- (0,0)  node [above, label=right: $\bar{1}$] {};
\draw[-] (1,-1) -- (1,-2);
\draw[-] (1,-1) -- (2,0) node [above, label=left:1] {};
\node[circle,draw=black, fill=white, inner sep=0pt,minimum size=4pt] (a) at (0,0) {};
\node[circle,draw=black, fill=white, inner sep=0pt,minimum size=4pt] (b) at (2,0) {};
\node (d) at (1.2,-1.5) {$0$};
\end{tikzpicture}\end{center}
\begin{center}
\begin{tikzpicture}[scale=.7] 
\node (a) at (0,0) {$\bullet$};
\node (f) at (2,-1) {$\bullet$};
\node(g) at (1.5,-1.4) {$\bar{1}$};
\draw[-] (1,-1) -- (0,0)  node [above, label=right: 1] {};
\draw[-] (1,-1) -- (0,-2) node [below, label=right:0] {};
\draw[-] (1,-1) -- (2,-1) node {};
\draw[-] (2,-1) -- (3,0) node [above, label=left:0] {};
\draw[-] (2,-1) -- (3,-2) node [below, label=left:1] {};
\node[circle,draw=black, fill=white, inner sep=0pt,minimum size=4pt] (b) at (3,0) {};
\node[circle,draw=black, fill=white, inner sep=0pt,minimum size=4pt] (e) at (1,-1) {};
\end{tikzpicture}
\hspace{.3in}
\begin{tikzpicture}[scale=.7] 
\node (a) at (0,0) {$\bullet$};
\node (f) at (2,-1) {$\bullet$};
\node(g) at (1.5,-1.4) {$\bar{1}$};
\draw[-] (1,-1) -- (0,0)  node [above, label=right: 0] {};
\draw[-] (1,-1) -- (0,-2) node [below, label=right: 1] {};
\draw[-] (1,-1) -- (2,-1) node {};
\draw[-] (2,-1) -- (3,0) node [above, label=left:0] {};
\draw[-] (2,-1) -- (3,-2) node [below, label=left:1] {};
\node[circle,draw=black, fill=white, inner sep=0pt,minimum size=4pt] (b) at (3,0) {};
\node[circle,draw=black, fill=white, inner sep=0pt,minimum size=4pt] (e) at (1,-1) {};
\end{tikzpicture}
\hspace{.3in}
\begin{tikzpicture}[scale=.7] 
\node (a) at (0,0) {$\bullet$};
\node (f) at (2,-1) {$\bullet$};
\node(g) at (1.5,-1.4) {$\bar{1}$};
\draw[-] (1,-1) -- (0,0)  node [above, label=right: 0] {};
\draw[-] (1,-1) -- (0,-2) node [below, label=right:1] {};
\draw[-] (1,-1) -- (2,-1) node {};
\draw[-] (2,-1) -- (3,0) node [above, label=left: 1] {};
\draw[-] (2,-1) -- (3,-2) node [below, label=left:0] {};
\node[circle,draw=black, fill=white, inner sep=0pt,minimum size=4pt] (b) at (3,0) {};
\node[circle,draw=black, fill=white, inner sep=0pt,minimum size=4pt] (e) at (1,-1) {};
\end{tikzpicture}
\hspace{.3in}
\raisebox{.5in}{\begin{tikzpicture}[scale=.8] 
\node (a) at (0,0) {$\bullet$};
\draw[-] (0,0)  to[out=-50,in=230] (2,0)  node [below, label=right:1] {};
\draw[-] (2,0)  to[out=230,in=-50] (0,0)  node [below, label=left:1] {};
\node[circle,draw=black, fill=white, inner sep=0pt,minimum size=4pt] (b) at (2,0) {};
\end{tikzpicture}} 
\end{center}
\begin{center}
\begin{tikzpicture}[scale=.7] 
\node (a) at (3,0) {$\bullet$};
\node (f) at (1,-1) {$\bullet$};
\node(g) at (1.5,-1.4) {$1$};
\draw[-] (1,-1) -- (0,0)  node [above, label=right: $\bar{1}$] {};
\draw[-] (1,-1) -- (0,-2) node [below, label=right:0] {};
\draw[-] (1,-1) -- (2,-1) node {};
\draw[-] (2,-1) -- (3,0) node [above, label=left:0] {};
\draw[-] (2,-1) -- (3,-2) node [below, label=left:
$\bar{1}$] {};
\node[circle,draw=black, fill=white, inner sep=0pt,minimum size=4pt] (b) at (0,0) {};
\node[circle,draw=black, fill=white, inner sep=0pt,minimum size=4pt] (e) at (2,-1) {};
\end{tikzpicture}
\hspace{.3in}
\begin{tikzpicture}[scale=.7] 
\node (a) at (3,0) {$\bullet$};
\node (f) at (1,-1) {$\bullet$};
\node(g) at (1.5,-1.4) {$1$};
\draw[-] (1,-1) -- (0,0)  node [above, label=right: 0] {};
\draw[-] (1,-1) -- (0,-2) node [below, label=right:$\bar{1}$] {};
\draw[-] (1,-1) -- (2,-1) node {};
\draw[-] (2,-1) -- (3,0) node [above, label=left:0] {};
\draw[-] (2,-1) -- (3,-2) node [below, label=left:$\bar{1}$] {};
\node[circle,draw=black, fill=white, inner sep=0pt,minimum size=4pt] (b) at (0,0) {};
\node[circle,draw=black, fill=white, inner sep=0pt,minimum size=4pt] (e) at (2,-1) {};
\end{tikzpicture}
\hspace{.3in}
\begin{tikzpicture}[scale=.7] 
\node (a) at (3,0) {$\bullet$};
\node (f) at (1,-1) {$\bullet$};
\node(g) at (1.5,-1.4) {$1$};
\draw[-] (1,-1) -- (0,0)  node [above, label=right: 0] {};
\draw[-] (1,-1) -- (0,-2) node [below, label=right:$\bar{1}$] {};
\draw[-] (1,-1) -- (2,-1) node {};
\draw[-] (2,-1) -- (3,0) node [above, label=left:$\bar{1}$] {};
\draw[-] (2,-1) -- (3,-2) node [below, label=left:0] {};
\node[circle,draw=black, fill=white, inner sep=0pt,minimum size=4pt] (b) at (0,0) {};
\node[circle,draw=black, fill=white, inner sep=0pt,minimum size=4pt] (e) at (2,-1) {};
\end{tikzpicture}
\hspace{.3in}
\raisebox{.5in}{\begin{tikzpicture}[scale=.8] 
\node (a) at (2,0) {$\bullet$};
\draw[-] (0,0)  to[out=-50,in=230] (2,0)  node [below, label=right:$\bar{1}$] {};
\draw[-] (2,0)  to[out=230,in=-50] (0,0)  node [below, label=left:$\bar{1}$] {};
\node[circle,draw=black, fill=white, inner sep=0pt,minimum size=4pt] (b) at (0,0) {};
\end{tikzpicture}} 
\end{center}
\caption{Khovanov--Kuperberg's inductive web growth rules.}
\label{fig:growthrules}
\end{figure}

\begin{remark}\label{rem:tricolor}
We use the slightly modified but equivalent web growth rules of \cite{petersen2009promotion}. Notice that using this version of the growth rules, each interior vertex in the resulting web is adjacent to exactly one edge with each label. We will use this property in later sections.
\end{remark}

\begin{figure}[h!]
\begin{center}
\includegraphics[width=2in]{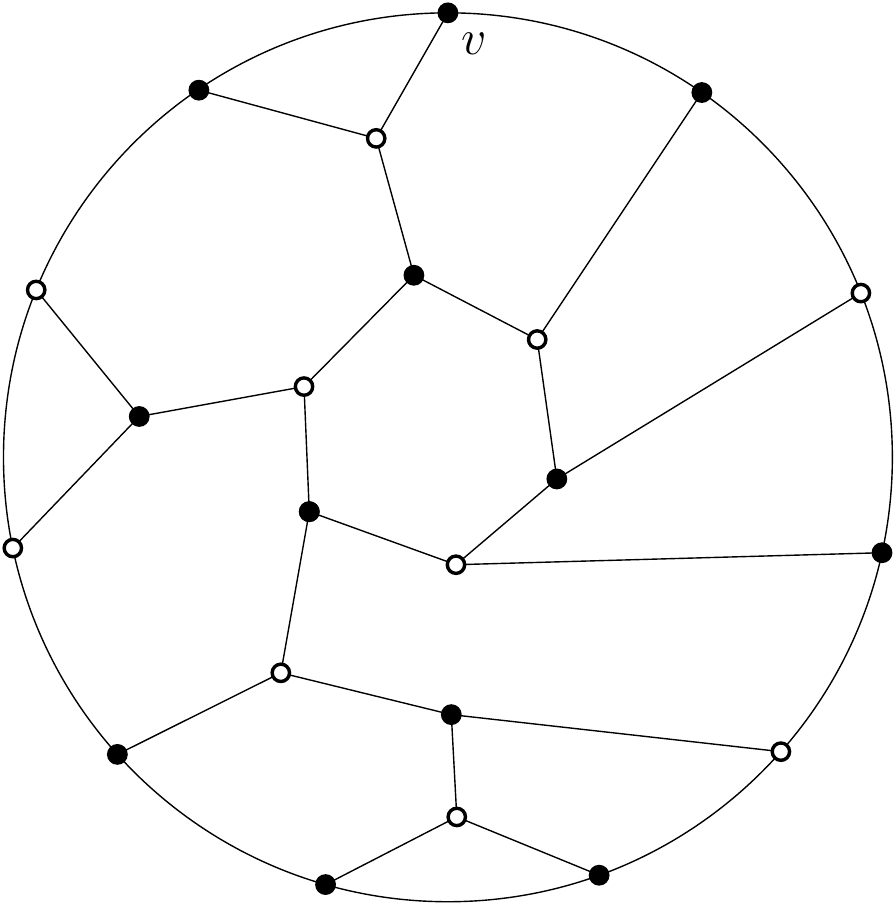}\hspace{1in}
\includegraphics[width=1.8in]{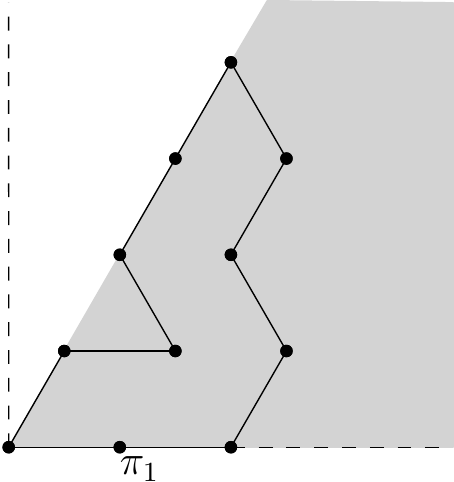}\vspace{.3in}
\includegraphics[width=6in]{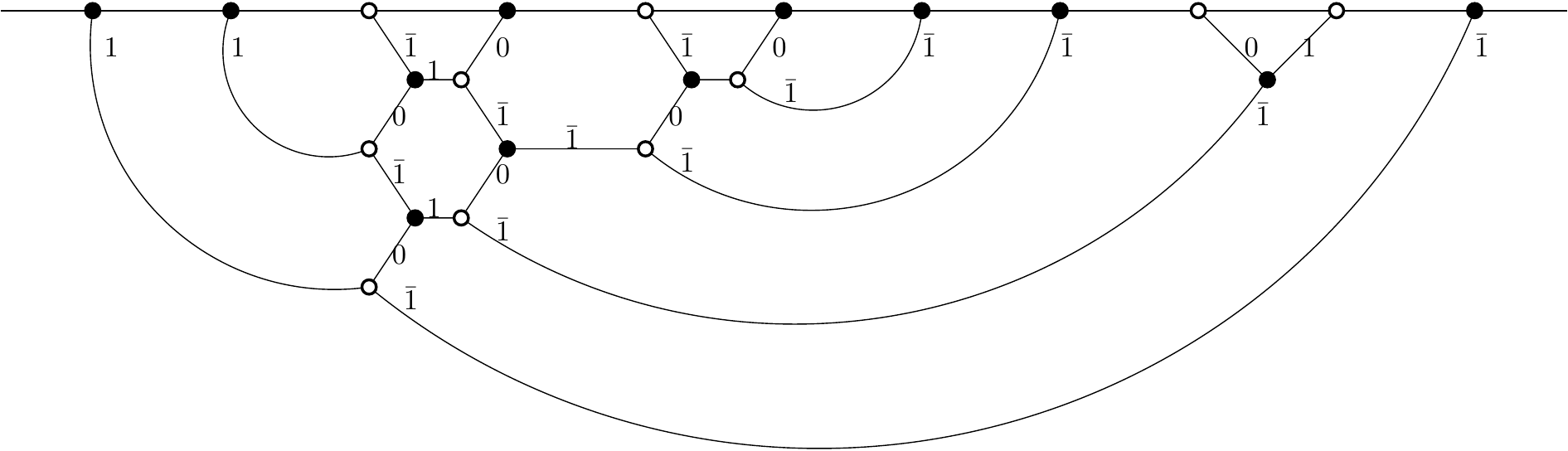}
\end{center}
\caption{The top left image is a web $D$ with marked chosen leftmost vertex $v$. The top right image shows the corresponding dominant path in the dominant Weyl chamber. The image at the bottom is an illustration of the web growth rules beginning with the signature and state string for $D$, which can easily be read from the image.}
\label{fig:growthfigure}
\end{figure}

Khovanov and Kuperberg prove the following results.

\begin{proposition}\cite{khovanov1999web}
\begin{enumerate}
\item The web built from a signature and state string does not depend on the order in which the growth rules are carried out.
\item The growth algorithm on a dominant signature and state string does not terminate until the strings have length 0.
\end{enumerate}
\end{proposition}

This growth algorithm has an inverse for dominant signature and state strings. Given a web, Khovanov and Kuperberg describe how to construct its corresponding signature and state string using a \textit{minimal cut path algorithm}, which is inverse to the growth rules. We do not describe this algorithm here but use its existence to emphasize the fact that there is a one-to-one correspondence between (irreducible, non-elliptic) webs with a chosen leftmost vertex and dominant signature and state strings. 

Given a web $D$ with chosen leftmost vertex $v_1$ and boundary vertices read in clockwise order $v_1,\ldots,v_n$, we say that the \textit{state of boundary vertex $v_i$} is the state in the $i$th component of the corresponding signature and state string. For the web in Figure~\ref{fig:growthrules} with indicated leftmost vertex, we have that the state of $v_3$ is $\bar{1}$ and the state of $v_4$ is 0.

We call the web consisting of one edge between a black boundary vertex and a white boundary vertex the \textit{identity web}. In future sections, we will wish to restrict our attention to webs that do not contain the identity web as a disjoint component. See Figure~\ref{fig:identityweb}. 

\begin{figure}[h!]
\begin{center}
\includegraphics[width=3in]{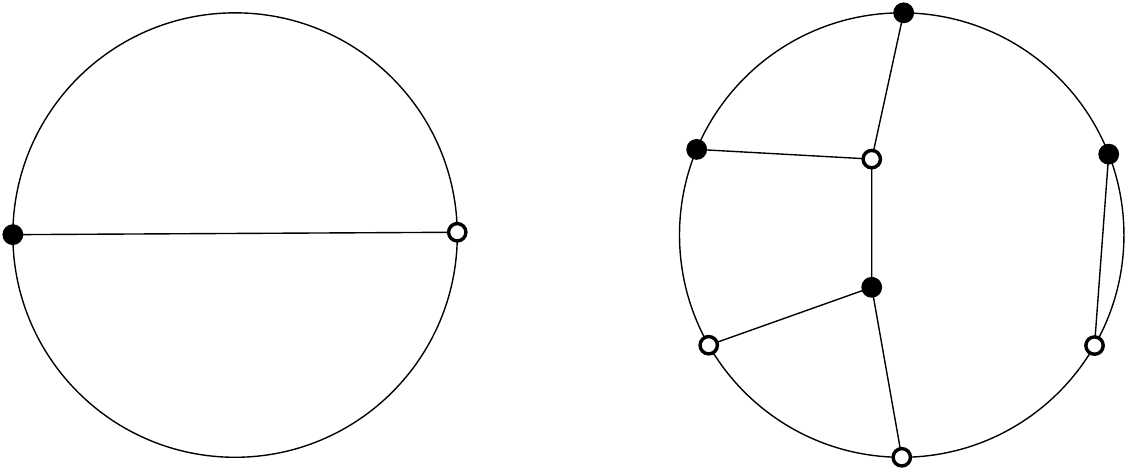}
\end{center}
\caption{The identity web and a web containing the identity web as a disjoint component.}
\label{fig:identityweb}
\end{figure}

\begin{proposition}\label{prop:hitextremerays}
Suppose $D$ is a web with $n$ boundary vertices and chosen leftmost vertex $v$ that does not contain the identity web. Denote the corresponding dominant path by $\pi=(\pi_0,\pi_1,\ldots,\pi_n)$, where $\pi_0=\pi_n$ is the origin. 
\begin{enumerate}
\item If $v$ is black, then the first $\pi_k$ with $k>0$ that lies on the upper extreme ray does not also lie on the lower extreme ray.
\item If $v$ is white, then the first $\pi_k$ with $k>0$ that lies on the lower extreme ray does not also lie on the upper extreme ray.
\end{enumerate}
\end{proposition}

\begin{proof}

First assume $v$ is black, and denote the signature and state string for $D$ by $((j_1,s_1),\ldots,(j_n,s_n))$. Suppose to the contrary that none of $\pi_2,\ldots,\pi_{n-1}$ lie on the upper extreme ray and not the lower extreme ray. Since $\pi$ ends at the origin, we know there is some first $\pi_\ell$ that lies on both extreme rays, i.e., on the origin, with $0<\ell\leq n$. Consider the dominant path $\pi'=(\pi_0,\ldots,\pi_\ell)$. We know that the first step of $\pi'$ is in direction $(1,\bullet)$ and the last step must be in direction $(1,\circ)$, otherwise $\pi_{\ell-1}$ would lie on only the upper extreme ray. It follows that we can horizontally translate path $\pi''=(\pi_1,\ldots,\pi_{\ell-1})$ one unit left and obtain another dominant path: $(\pi_1-\mu_1,\ldots,\pi_{\ell-1}-\mu_1)$. Therefore applying the web growth rules to the corresponding section of the signature and state string for $D$, $((j_2,s_2),\ldots,(j_{\ell-1},s_{\ell-1}))$, yields a web $D'$ contained in $D$. Applying the web growth rules to $((j_1,s_1),\ldots,(j_\ell,s_\ell))$ hence gives a disjoint union of $D'$ with the identity web. We have thus shown that $D$ contains the identity web, which is a contradiction.

A similar argument holds if $v$ is white.
\end{proof}

\subsection{Webs in $\mathcal{M}_n$ and tableaux}
Let $\mathcal{M}_n$ denote the set of  webs with $3n$ boundary vertices, all of which are the same color. We will assume that all boundary vertices are black. Note that in this setting, we may ignore the signature of a web in $D\in\mathcal{M}_n$ with chosen leftmost vertex and instead focus on its state string. We consider its state string to be a word $w(D)$ in the alphabet $\{1,0,\bar{1}\}$, and we refer to the word as $w$ if the corresponding web is clear from context.

Using $w(D)=w_1\cdots w_{3n}$ for a web $D\in\mathcal{M}_n$, we can easily associate a standard Young tableau $T(w(D))$ of shape $(n,n,n)$ as follows. Fill the top row of $T(w(D))$ with the indices of the 1's in $w(D)$, the second row of $T(w(D))$ with the indices of the 0's in $w(D)$, and the third row with the indices of the $\bar{1}$'s of $w(D)$. For example, the word $w(D)=1101\bar{1}00\bar{1}\bar{1}$ corresponds to the standard Young tableau below.
\begin{center}
\raisebox{-.1in}{$T(w(D))=$}
\begin{ytableau}
1 & 2 & 4 \\
3 & 6 & 7 \\
5 & 8 & 9
\end{ytableau}
\end{center}

\section{Rotation of webs in $\mathcal{M}_n$ and promotion}\label{sec:blackpromotionandrotation}
We now review the results of T.K. Peterson, P. Pylyavksyy, and B. Rhoades for webs whose boundary vertices are all the same color. In \cite{petersen2009promotion}, the authors reinterpret the action of promotion on rectangular standard Young tableaux of shape $(n,n,n)$ as counterclockwise rotation of webs in $\mathcal{M}_n$. We will again assume all boundary vertices are black. In what follows, if $D$ is a web with chosen leftmost vertex $v$, let $p(D)$ denote the web with chosen leftmost vertex the next vertex reached traveling clockwise around the boundary from $v$. In other words, $p(D)$ is the result of rotating $D$ one vertex counterclockwise. 

We first define the \textit{left cut} and \textit{right cut} starting at boundary vertex $v$ of a web $D$. 

\begin{definition}[\cite{petersen2009promotion}]
The \textit{left cut} starting at $v$ is defined to be the path starting from $v$ obtained by first traveling along the edge adjacent to $v$ away from the boundary, turning left at the first vertex encountered, right at the next vertex encountered, left at the third vertex, etc. until we reach the boundary of $D$. The \textit{right cut} starting at $v$ is defined analogously starting with a right turn at the first vertex encountered, followed by a left turn, etc. until we reach the boundary of $D$.
\end{definition}

Let $v^l$ denote the boundary vertex at the end of the left cut starting at $v$ and $v^r$ denote the boundary vertex at the end of the right cut starting at $v$.

\begin{lemma}[Lemma 3.2 \cite{petersen2009promotion}]\label{lem:cutsdon'tintersect}
For any $D\in\mathcal{M}_n$ and for any boundary vertex $v$ of $D$, the left and right cuts do not intersect each other after the initial edge and do not self-intersect. In other words, $v$, $v^l$, and $v^r$ are distinct.
\end{lemma}

\begin{lemma}[Lemma 3.6 \cite{petersen2009promotion}]\label{lem:balanced} Let $D\in\mathcal{M}_n$ with leftmost vertex $v$, and consider vertices $v^l$ and $v^r$ obtained from the left cut and right cut starting at $v$.
\begin{enumerate}
\item Among the states of vertices preceding $v^l$ (inclusively), there are as many 1's as 0's. Further, $v^l$ is the leftmost vertex with this property.
\item Among the states preceding $v^r$ (inclusively), there are as many 0's as $\bar{1}$'s. Further, $v^r$ is the leftmost vertex to the right of $v^l$ with this property.
\end{enumerate}
\end{lemma}

We can easily rephrase Lemma~\ref{lem:balanced} in terms of the dominant paths. The vertex $v^l$ corresponds to the first place that the dominant path returns to the upper extreme ray of the dominant chamber and $v^r$ corresponds to the first place after $v^l$ that the dominant path returns to the lower extreme ray. 

\begin{theorem}\label{thm:allblackwordchange}\cite{petersen2009promotion}
Let $D\in\mathcal{M}_n$ be a web with chosen leftmost vertex and corresponding word $w=w_1w_2\cdots w_{3n}$. Suppose $v^l$ corresponds to $w_a$ and $v^r$ to $w_b$. Then the word obtained after rotating $D$ one vertex counterclockwise, $w'=w(p(D))$, is 
\[w'=w_2\cdots w_{a-1}1w_{a+1}\cdots w_{b-1}0w_{b+1}\cdots w_{3n}\bar{1}.\]
\end{theorem}

In other words, upon one rotation counterclockwise, the labels of the web $D$ change only at three vertices, the leftmost vertex $v$, $v^l$, and $v^r$. These vertices can be identified either as 
\begin{itemize}
\item $v$: the leftmost vertex,
\item $v^l$: the first place there are as many 0's as 1's is $w(D)$, and
\item $v^r$: the place after $v^l$ there are as many $\bar{1}$'s as 0's in $w(D)$,
\end{itemize}
or as 
\begin{itemize}
\item $v$: the leftmost vertex,
\item $v^l$: the first place the corresponding dominant path returns to the upper extreme ray of the dominant chamber, and 
\item $v^r$: the first place after $v^l$ the corresponding dominant path returns to the lower extreme ray of the dominant chamber.
\end{itemize}

The first interpretation leads more easily to the result below after thinking about which entries in a tableau move upward during promotion, while the second interpretation is the one we will generalize in the case that not all boundary vertices of $D$ are the same color. Notice that the first interpretation implies that the state of $v^l$ in $D$ is 0 and the state of $v^r$ is $\bar{1}$ in $D$.

\begin{corollary}\cite{petersen2009promotion}
For any web $D$ with fixed leftmost vertex and with black boundary vertices, we have
\[T(w(p(D)))=p(T(w(D))).\] That is, the tableau associated  with the rotation of $D$ is given by promotion of the tableau associated with $D$ itself.
\end{corollary}

\begin{remark}
If we instead assume that a web  $D\in\mathcal{M}_n$ has all white boundary vertices, we see that 
\[w'=w_2\cdots w_{a-1}\bar{1}w_{a+1}\cdots w_{b-1}0w_{b+1}\cdots w_{3n}1.\]
In this case, $v^l$ can be described either as the first place there are as many 0's as $\bar{1}$'s or the first place the corresponding dominant path returns to the lower extreme ray of the dominant chamber, and $v^r$ can be described as the first place to the right of $v^l$ where there are as many $1$'s as $0$'s or the first place after $v^l$ the corresponding dominant path returns to the upper extreme ray of the dominant chamber. In particular, $v^l$ has state 0 and $v^r$ has state $\bar{1}$.
\end{remark}

\section{Oscillating Tableaux and Promotion}\label{sec:osctablandpromotion}

An \textit{oscillating tableau} of length $k$ is a sequence of $k+1$ partitions $(\lambda^0=\emptyset, \lambda^1,\ldots,\lambda^k)$ such that $\lambda^i$ is obtained from $\lambda^{i-1}$ by either adding or removing one box. A standard Young tableau is an example of an oscillating tableau where one box is added at each step. We now both generalize and restrict this notion by allowing $\lambda^i$ to have negative part sizes and parts of size 0 and by fixing that each $\lambda^i$ has $n$ parts for some fixed $n$.

We define a \textit{generalized partition} to be a finite non-increasing sequence of integers $\lambda=(\lambda_1\geq \lambda_2\geq\cdots\geq \lambda_n)$. We say that $\lambda$ is a \textit{generalized partition with $n$ parts} if it is possible to write $\lambda$ with $n$ components, where we are allowed to add zeros at the end of the partition when doing so preserves the weakly decreasing property of a generalized partition. For example, $(2,1)$ is a generalized partition with $n$ parts for any $n\geq 2$, as we consider $(2,1)$ to be the same as $(2,1,0)$, etc. We consider the empty partition $\emptyset$ to be the same as $(0,\ldots,0)$, and so $\emptyset$ can be written with any number of parts. However, $(5,5,3,0,-2,-2)$ is a generalized partition with 6 parts and cannot be written with a different number of parts. 

To each generalized partition, we associate a \textit{generalized Young diagram} by considering the generalized Young diagram to lie in the lower half plane of $\mathbb{Z}^2$ and allowing boxes to lie in both the third and fourth quadrant. Boxes corresponding to negative parts of the generalized partition are in the third quadrant and boxes corresponding to the positive parts are in the fourth quadrant. The generalized partition $(5,5,3,0,-2,-2)$ has generalized Young diagram shown below, where we color boxes in the third quadrant (i.e., boxes corresponding to negative part sizes) red.

\begin{center}
\ytableausetup{boxsize=.15in}
\begin{ytableau}
\none & \none & & & & & \\
\none & \none & & & & & \\
\none & \none & & &  \\
\none \\
*(red) & *(red)\\
*(red) & *(red)
\end{ytableau}
\end{center}

\begin{definition}
A \textit{generalized oscillating tableau} of length $k$ with $n$ parts is a sequence of $k+1$ generalized partitions $(\lambda^0=\emptyset, \lambda^1,\ldots,\lambda^k)$ such that $\lambda^i$ is obtained from $\lambda^{i-1}$ by either adding or removing one box and such that each $\lambda^i$ is a generalized partition with $n$ parts. We denote the set of such sequences by GOT$(k,n)$.
\end{definition}

To each element of GOT$(k,n)$, we can associate a set-valued filling of the union of boxes appearing in $\lambda^1,\ldots,\lambda^k$ by entering $i$ or $i'$ in the box that was added to or removed from, respectively, $\lambda^{i-1}$ to obtain $\lambda^i$. Within each box, we write the subset in increasing order. We identify a generalized oscillating tableau with the corresponding set-valued filling. Note that the shape of the generalized oscillating tableau need not be a generalized partition shape. We again color cells corresponding to negative cells red. We refer to the red part of the filling as the \textit{negative part} and the white part as the \textit{positive part}.

\begin{remark}
Note that it is not necessary to use both primed and unprimed entries in this construction as the location of the primed entries can easily be recovered from the analogous construction using only unprimed entries. Indeed, the entries in odd positions on the negative side and even positions on the positive side are exactly the entries with primes. However, we will use primed and unprimed entries for ease of understanding and notational convenience. 
\end{remark}

\begin{example}\label{ex:genosctab}
The generalized oscillating tableau \[(\emptyset, (1,0,0), \emptyset, (0,0,-1), (0,0,-2), (1,0,-2), (1,1,-2), (1,1,-1))\in\text{GOT}(7,3)\] corresponds to the filling on the left of shape $(1,1,-2)$. The sequence
\[(\emptyset,(0,0,-1), (0,-1,-1), (1,-1,-1), (1,0,-1),(1,0,-2),(1,1,-2),(1,0,-2),(1,0,-1),(1,0,0))\in \text{GOT}(9,3)\] corresponds to the filling on the right, which the does have generalized partition shape.
In contrast, the sequence
\[(\emptyset, (-1), \emptyset, (1), (2), (2,1))\] is not in GOT$(5,n)$ for any $n$ since the partition $(-1)$ can have only one part and the partition $(2,1)$ must have at least two parts.
\begin{center}
\ytableausetup{boxsize=.3in}
\begin{ytableau}
\none & \none & 12'5\\
\none & \none & 6 \\
*(red) 4'7  & *(red) 3'
\end{ytableau}\hspace{1in}
\begin{ytableau}
\none & \none & 3 \\
\none & *(red) 2'4 & 67' \\
*(red) 5'8  &*(red) 1'9
\end{ytableau}
\end{center}
\end{example}

Although a generalized oscillating tableau may not be of generalized partition shape, we can say something about its shape. In the proposition below, the size of the rows of the positive (resp. negative) part of generalized oscillating tableau $T$ is the number of white (resp. red) boxes in that row of $T$. For example, the tableau on the right in Example~\ref{ex:genosctab} has rows on the positive side of size 1,1, and 0, while the rows on the negative side have size 0, 1, and 2. 

\begin{proposition}
Let $T\in\text{GOT}(k,n)$. Then the positive part of $T$ has rows of weakly decreasing size, and the negative part of $T$ has rows of weakly increasing size.
\end{proposition}

\begin{proof}
The result follows from the fact that the shape of $T$ is the union of $k+1$ generalized partition shapes with $n$ parts. Thus if cell $(i,j)\in T$ for $i>1,j>0$, $(i,j)\in\lambda^s$ for some $1\leq s \leq k$. Since $\lambda^s$ is a generalized partition, $(i-1,j)\in\lambda^s$ and thus $(i-1,j)\in T$. Similarly, if cell $(i,j)\in T$ for $i<n,j<0$, $(i,j)\in\lambda^s$ for some $1\leq s \leq k$. Since $\lambda^s$ is a generalized partition, $(i+1,j)\in\lambda^s$ and thus $(i+1,j)\in T$.
\end{proof}

\subsection{Webs and generalized oscillating tableaux}
Given a web $D$ with chosen leftmost vertex $v$ and $k$ boundary vertices, we may associate a generalized oscillating tableau of length $k$ with three parts as follows. First compute the signature and state string corresponding to $D$. The pairs in the signature and state string will correspond to the following actions on generalized partition $\lambda^i$.

\begin{center}
\begin{tabular}{|c| c|}
\hline 
\textcolor{white}{\huge{l}}$(1,\bullet)$\textcolor{white}{\huge{l}} & Add one to the first part of $\lambda^i$ to obtain $\lambda^{i+1}$\\
\hline
\textcolor{white}{\huge{l}}$(0,\bullet)$\textcolor{white}{\huge{l}} & Add one to the second part of $\lambda^i$ to obtain $\lambda^{i+1}$ \\
\hline
\textcolor{white}{\huge{l}}$(\bar{1},\bullet)$\textcolor{white}{\huge{l}} & Add one to the third part of $\lambda^i$ to obtain $\lambda^{i+1}$\\
\hline 
\textcolor{white}{\huge{l}}$(1,\circ)$\textcolor{white}{\huge{l}} & Subtract one from the first part of $\lambda^i$ to obtain $\lambda^{i+1}$ \\
\hline
\textcolor{white}{\huge{l}}$(0,\circ)$\textcolor{white}{\huge{l}} & Subtract one from the second part of $\lambda^i$ to obtain $\lambda^{i+1}$\\
\hline
\textcolor{white}{\huge{l}}$(\bar{1},\circ)$\textcolor{white}{\huge{l}} & Subtract one from the third part of $\lambda^i$ to obtain $\lambda^{i+1}$\\
\hline
\end{tabular}
\end{center}

Using the signature and state string, build the generalized oscillating tableau by reading the string left to right and using the table above to determine how to obtain the next generalized partition in the generalized oscillating tableau. For example, the tableau on the right in Example~\ref{ex:genosctab} comes from the signature and state string

\[((\bar{1},\circ), (0,\circ), (1,\bullet), (0,\bullet), (\bar{1},\circ), (0,\bullet), (0,\circ), (\bar{1},\bullet), (\bar{1},\bullet)).\]

It is clear that this procedure applied to a web will indeed produce a sequence of generalized partitions since dominant paths produce generalized partitions. This is because $\lambda^i_j$ gives the net number of steps taken in the $(j,\bullet)$ direction (i.e., the number of steps in direction $(j,\bullet)$ minus the number of steps in direction $(j,\circ)$) after $i$ steps in the dominant path, and being inside the dominant chamber is equivalent to saying that the net number of steps in direction $(j,\bullet)$ is weakly greater than the net number of steps in direction $(j+1,\bullet)$ for $j=1,2$ at each point in the path.

The fact that webs correspond to dominant signature and state strings gives us the following fact.

\begin{proposition}\label{prop:rowssamesize}
A generalized oscillating tableau $T=(\lambda^0,\ldots,\lambda^k)\in\text{GOT}(k,3)$ comes from a web $D$ if and only if $\lambda^k=(m,m,m)$ for some $m\in \mathbb{Z}$.
\end{proposition}
\begin{proof}
The condition that $\lambda^0=\emptyset$ and $\lambda^k=(m,m,m)$ is equivalent to the condition that the dominant path starts and ends at the origin.
\end{proof}

\begin{lemma}
Let $T$ be a generalized oscillating tableau coming from web $D$. Then $p(T)$ corresponds to a web.
\end{lemma}
\begin{proof}
The result follows from Proposition~\ref{prop:rowssamesize} and the fact that $\lambda^k=\mu^k$ in generalized oscillating promotion.
\end{proof}

\subsection{Generalized Oscillating Promotion Growth rules}\label{sec:genoscprogrowth}
We first define a promotion action $p:\text{GOT}(k,n)\to\text{GOT}(k,n)$ using growth rules. We call this action \textit{generalized oscillating promotion}. It is important to note that in contrast to promotion on standard Young tableaux, $p(T)$ may not have the same shape as generalized oscillating tableau $T$. 

Let $T=(\lambda^0, \lambda^1,\ldots,\lambda^k)\in\text{GOT}(k,n)$. We inductively build a new sequence of generalized partition shapes $(\mu^0=\emptyset, \mu^1,\ldots,\mu^k)$ using a set of growth rules as in Section~\ref{sec:tabandprom}. The sequence we create gives $p(T)$.

The idea of the growth rules is this: Suppose we perform action 1 to obtain $\lambda^s$ from $\mu^{s-1}$ and action 2 to obtain $\lambda^{s+1}$ from $\lambda^s$. Then if performing action 2 on $\mu^{s-1}$ gives a generalized partition with $n$ parts, $\mu^s$ is obtained from $\mu^{s-1}$ by performing action 2. If performing action 2 on $\mu^{s-1}$ does not give a generalized partition with $n$ parts, we then either obtain $\mu^s$ from $\mu^{s-1}$ by performing action 1 or we, in two cases, must otherwise modify. See Figure~\ref{fig:growthschematic} for a schematic of this idea. The precise rules are as follows.

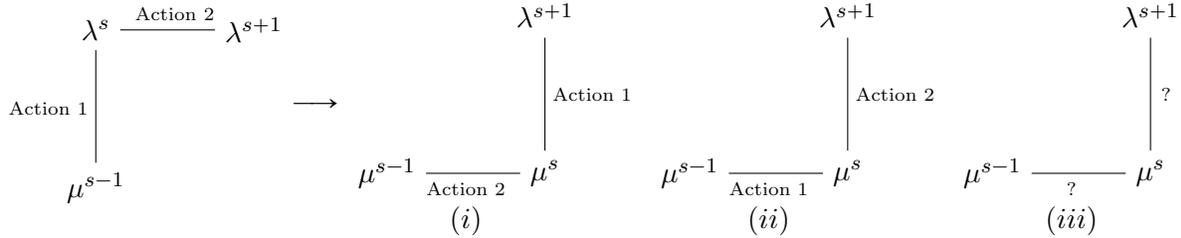
\begin{figure}[h!]
\begin{tikzpicture}[scale=2.1]
\node (1) at (0,0) {$\lambda^s$};
\node (2) at (1,0) {$\lambda^{s+1}$};
\node (3) at (0,-1) {$\mu^{s-1}$};
\node (4) at (.5,.1) {\tiny{Action 2}};
\node (5) at (-.3,-.5) {\tiny{Action 1}};
\node (6) at (.5,-1.3) {};
\draw (1)--(2) (1)--(3);
\end{tikzpicture}\raisebox{.7in}{$\longrightarrow$}
\begin{tikzpicture}[scale=2.1]
\node (2) at (1,0) {$\lambda^{s+1}$};
\node (3) at (0,-1) {$\mu^{s-1}$};
\node (4) at (1,-1) {$\mu^s$};
\node (5) at (.5,-1.1) {\tiny{Action 2}};
\node (6) at (1.3, -.5) {\tiny{Action 1}};
\node (7) at (.5,-1.3) {$(i)$};
\draw (2)--(4) (3)--(4);
\end{tikzpicture}
\begin{tikzpicture}[scale=2.1]
\node (2) at (1,0) {$\lambda^{s+1}$};
\node (3) at (0,-1) {$\mu^{s-1}$};
\node (4) at (1,-1) {$\mu^s$};
\node (5) at (.5,-1.1) {\tiny{Action 1}};
\node (7) at (1.3, -.5) {\tiny{Action 2}};
\node (6) at (.5,-1.3) {$(ii)$};
\draw (2)--(4) (3)--(4);
\end{tikzpicture}
\begin{tikzpicture}[scale=2.1]
\node (2) at (1,0) {$\lambda^{s+1}$};
\node (3) at (0,-1) {$\mu^{s-1}$};
\node (4) at (1,-1) {$\mu^s$};
\node (5) at (.5,-1.1) {\tiny{?}};
\node (7) at (1.1, -.5) {\tiny{?}};
\node (6) at (.5,-1.3) {$(iii)$};
\draw (2)--(4) (3)--(4);
\end{tikzpicture}
\caption{The idea of the generalized oscillating promotion growth rules. If the result of $(i)$ is a generalized partition with the correct number of parts, we choose $(i)$. Otherwise, we choose $(ii)$ in (OP1b) and (OP2b). Rules (OP1d) and (OP2d) correspond to $(iii)$, where we modify depending on $\mu^{s-1}$.}
\label{fig:growthschematic}
\end{figure}

\begin{definition}(Generalized oscillating promotion growth rules)
\begin{enumerate}
\item[(OP1)] Suppose $\lambda^s$ is obtained from $\mu^{s-1}$ by adding a box in row $i$.
\begin{enumerate}
\item If $\lambda^{s+1}$ is obtained from $\lambda^s$ by adding a box in row $j$ and the result of adding a box to row $j$ of $\mu^{s-1}$ is a generalized partition with $n$ parts, then $\mu^s$ is obtained from $\mu^{s-1}$ by adding a box in row $j$.
\item If $\lambda^{s+1}$ is obtained from $\lambda^s$ by adding a box in row $j\neq i$ and the result of adding a box to row $j$ of $\mu^{s-1}$ is not a generalized partition with $n$ parts, then $\mu^s$ is obtained from $\mu^{s-1}$ by adding a box in row $i$.
\item If $\lambda^{s+1}$ is obtained from $\lambda^s$ by deleting a box in row $j$ and deleting a box in row $j$ of $\mu^{s-1}$ is a generalized partition with $n$ parts, then $\mu^s$ is the result of deleting a box in row $j$ of $\mu^{s-1}$.
\item If $\lambda^{s+1}$ is obtained from $\lambda^s$ by deleting a box in row $i$ and deleting a box in row $i$ of $\mu^{s-1}$ is not a generalized partition with $n$ parts, then $\mu^s$ is the result of deleting a box of $\mu^{s-1}$ in row $i+t$ for $t>0$ as small as possible.
\end{enumerate}

\item[(OP2)] Suppose $\lambda^s$ is obtained from $\mu^{s-1}$ by deleting a box in row $i$.
\begin{enumerate}
\item If $\lambda^{s+1}$ is obtained from $\lambda^s$ by deleting a box in row $j$ and the result of deleting a box in row $j$ of $\mu^{s-1}$ is a generalized partition with $n$ parts, then $\mu^s$ is obtained from $\mu^{s-1}$ by deleting a box in row $j$.
\item If $\lambda^{s+1}$ is obtained from $\lambda^s$ by deleting a box in row $j\neq i$ and the result of deleting a box to row $j$ of $\mu^{s-1}$ is not a generalized partition with $n$ parts, then $\mu^s$ is obtained from $\mu^{s-1}$ by deleting a box in row $i$.
\item If $\lambda^{s+1}$ is obtained from $\lambda^s$ by adding a box in row $j$ and adding a box in row $j$ of $\mu^{s-1}$ is a generalized partition with $n$ parts, then $\mu^s$ is the result of adding a box in row $j$ of $\mu^{s-1}$.
\item If $\lambda^{s+1}$ is obtained from $\lambda^s$ by adding a box in row $i$ and adding a box in row $i$ of $\mu^{s-1}$ is not a generalized partition with $n$ parts, then $\mu^s$ is the result of adding a box in row $i-t$ of $\mu^{s-1}$, where $t>0$ is as small as possible.
\end{enumerate}
\item[(OP3)] $\mu^k=\lambda^k$.
\end{enumerate}
\end{definition}

\begin{remark}
Note that rules (OP1b) and (OP2b) need only be stated for $j\neq i$ because if $j=i$, the result is always a generalized partition with $n$ parts. Similarly, (OP1d) and (OP2d) need only be stated for $j=i$ because the situations described in (OP1c) and (OP2c) can fail to be generalized partitions with $n$ parts only when $j=i$.
\end{remark}

\begin{example}\label{ex:gopgrowth} Below is the growth diagram for promotion of the generalized oscillating tableau \[T=(\emptyset,(1,0,0), (2,0,0), (2,0,-1),(2,1,-1), (1,1,-1), (1,0,-1), (1,0,0), \emptyset).\] To construct $p(T)$, we use rules (OP1a) with $i=j=1$, (OP1c) with $i=1$ and $j=3$, (OP1a) with $i=1$ and $j=2$, (OP1d) with $i=1$ and $t=1$, (OP1c) with $i=j=2$, (OP1b) with $i=2$ and $j=3$, (OP1c) with $i=3$ and $j=1$, and finally (OP3).
\begin{center}
\ytableausetup{boxsize=.1in}
\begin{tikzpicture}[scale=1.6]
\node (0) at (0,0) {$\emptyset$};
\node (1) at (1,0) {\begin{ytableau} $ $ \end{ytableau}};
\node (2) at (2,0) {\begin{ytableau} $ $ & \end{ytableau}};
\node (3) at (3,0) {\begin{ytableau}\none & & \\ \none \\ *(red)  \end{ytableau}};
\node (4) at (4,0) {\begin{ytableau}\none & & \\ \none &  \\ *(red)  \end{ytableau}};
\node (5) at (5,0) {\begin{ytableau} \none & \\ \none & \\ *(red) \end{ytableau}};
\node (6) at (6,0) {\begin{ytableau} \none & \\ \none \\ *(red) \end{ytableau}};
\node (7) at (7,0) {\begin{ytableau} $ $ \end{ytableau}};
\node (8) at (8,0) {$\emptyset$};
\node (0') at (1,-1) {$\emptyset$};
\node (1') at (2,-1) {\begin{ytableau} $ $ \end{ytableau}};
\node (2') at (3,-1) {\begin{ytableau} \none &  \\ \none \\ *(red) \end{ytableau}};
\node (3') at (4,-1) {\begin{ytableau}\none & \\ \none &  \\ *(red)  \end{ytableau}};
\node (4') at (5,-1) {\begin{ytableau}\none &  \\ \none   \\ *(red)  \end{ytableau}};
\node (5') at (6,-1) {\begin{ytableau} \none & \\ *(red) \\ *(red) \end{ytableau}};
\node (6') at (7,-1) {\begin{ytableau} \none & \\ \none \\ *(red) \end{ytableau}};
\node (7') at (8,-1) {\begin{ytableau} \none & \none & *(red) \end{ytableau}};
\node (8') at (9,-1) {$\emptyset$};
\draw (0)--(1)--(2)--(3)--(4)--(5)--(6)--(7)--(8);
\draw (0')--(1')--(2')--(3')--(4')--(5')--(6')--(7')--(8');
\draw (1)--(0') (2)--(1') (3)--(2') (4)--(3') (5)--(4') (6)--(5') (7)--(6') (8)--(7');
\end{tikzpicture}
\end{center}
\end{example}

\begin{lemma}
The shape $\mu^s$ produced from generalized partitions with $n$ parts $\lambda^s$, $\lambda^{s+1}$, and $\mu^{s-1}$ is a generalized partition with $n$ parts.
\end{lemma}
\begin{proof}
The result is clear for all rules except (OP1d) and (OP2d). 

For rule (OP1d), we know there is some row $i-t$ from which we can delete a box and obtain a generalized partition because we can always delete a box from the last row. There is a parallel argument for (OP2d).

\end{proof}

The following lemma follows immediately from the promotion growth rules.
\begin{lemma}
The shape $\mu^s$ produced from generalized partitions with $n$ parts $\lambda^s$, $\lambda^{s+1}$, and $\mu^{s-1}$ differs by one box from $\lambda^{s+1}$.
\end{lemma}

\begin{lemma}\label{lem:labelsincrease}
\begin{itemize}
\item[$(a)$] If $\lambda^1=(1,0,\ldots,0)$, then $\lambda^s$ is always obtained from $\mu^{s-1}$ by adding a box in some row $i_s$, and
$1=i_1\leq i_2\leq\cdots\leq i_{k}$.

\item[$(b)$] If $\lambda^1=(0,\ldots,0,-1)$, then $\lambda^s$ is always obtained from $\mu^{s-1}$ by deleting a box in some row $i_s$, and
$n=i_1\geq i_2\geq\cdots\geq i_k$.
\end{itemize}
\end{lemma}
\begin{proof}
We know that $\lambda^1=\lambda^1/\mu^0$ is either $(1,0,\ldots,0)$ or $(0,\ldots,0,-1)$. Assume $\lambda^1=(1,0,\ldots,0)$. Then $\lambda^r/\mu^{r-1}$ is one block in the first row until we first apply one of rule (OP1b) or (OP1d). If we then apply either rule (OP1b) or (OP1d) to obtain $\mu^s$, we know that $\mu^{s-1}_1=\mu^{s-1}_2$. If we use (OP1b), we must have $i=1$ and $j=2$, and so $\lambda^{s+1}$ differs from $\mu^s$ by one box in the second row. If we use (OP1d), $\lambda^{s+1}$ differs from $\mu^s$ by one block in some row $(1+t)$ with $t>0$. In either case, from this point forward, $\lambda^r/\mu^{r-1}$ will be one block in some row $1+t$ with $t>0$ until we next apply one of rule (OP1b) or (OP1d). After applying this rule, $\lambda^r/\mu^{r-1}$ will be a box in some row $1+t+s$ for $s>0$. Continuing in this manner, gives the result. The argument for $\lambda^1=(0,\ldots,0,-1)$ is similar.
\end{proof}

\begin{lemma}\label{lem:jumpgivesequality}
\begin{itemize}
\item[$(a)$] Suppose $\lambda^1=(1,0,\ldots,0)$. Suppose $\lambda^s$ is obtained from $\mu^{s-1}$ by adding a box in row $i$, growth rule (OP1b) or (OP1d) is applied to $\mu^{s-1}$ to obtain $\mu^s$, and the previous application of growth rule (OP1b) or (OP1d) was at $\mu^{r-1}$. Then $\lambda^{s+1}$ is the first generalized partition of $T$ with $s+1>r+1$ with row $i$ the same size as row $i+1$. If $\mu^{s-1}$ is the first application of (OP1b) or (OP1d), then $\lambda^{s+1}$ is the first nonempty generalized partition of $T$ with $\lambda^{s+1}_1=\lambda^{s+1}_2$.

\item[$(b)$] If $\lambda^1=(0,\ldots,0,-1)$. Suppose $\lambda^s$ is obtained from $\mu^{s-1}$ by deleting a box in row $i$, growth rule (OP2b) or (OP2d) is applied to $\mu^{s-1}$ to obtain $\mu^s$, and the previous application of (OP2b) or (OP2d) was at $\mu^{r-1}$. Then $\lambda^{s+1}$ is the first partition of $T$ with $s+1>r+1$ with row $i$ the same size as row $i-1$. If $\mu^{s-1}$ is the first application of (OP2b) or (OP2d), then $\lambda^{s+1}$ is the first nonempty generalized partition of $T$ with $\lambda^{s+1}_n=\lambda^{s+1}_{n-1}$.
\end{itemize}
\end{lemma}

\begin{proof}
Suppose that $\lambda^1=(1,0,\ldots,0)$, and that $1=i_1=i_2=\cdots=i_\ell<i_{\ell+1}$. In other words, we applied rule (OP1b) or rule (OP1d) for the first time at $\mu^\ell$. Since we know that $\mu^\ell_1=\mu^\ell_2$, we also know from examining (OP1b) and (OP1d) that $\lambda^{\ell+2}_1=\lambda^{\ell+2}_2$. We claim that $\lambda^{\ell+2}$ is the first generalized partition of $T$ with this property. This follows from the fact that $i_{\ell+1}$ is the first time $i_r>1$. When $i_r=1$, $\lambda^{r+1}$ is obtained by adding 1 to the first row of $\mu^r$, and thus $\lambda^{r+1}_1\neq \lambda^{r+1}_2$.

We can similarly argue that the next time rule (OP1b) or (OP1d) is applied to obtain $\mu^{s-1}$ from $\mu^{s-2}$ corresponds to the first $\lambda^{s}$ with $s>\ell+2$ and $\lambda^{s}_{\ell+1}=\lambda^{s}_{\ell+2}$ because until we reach this point, $\lambda^r$ is obtained from $\mu^{r-1}$ by adding one box in row $\ell+1$. Continuing in this manner gives the desired result.
\end{proof}

We next prove the converse.

\begin{lemma}\label{lem:equalitygivesjumps} Let $i_s$ denote the row in which a row was added or from which a box was deleted from $\mu^{s-1}$ to obtain $\lambda^s$.
\begin{enumerate}
\item Suppose $\lambda^1=(1,0,\ldots,0)$ and $i_s=t$ for some $s\in [1,k]$. Suppose $\lambda^r$ is the first generalized partition with $r>s$ and $\lambda^r_t=\lambda^r_{t+1}$. Then $i_s=\cdots=i_{r-1}=t$ and $i_{r}>t$.
\item Suppose $\lambda^1=(0,\ldots,0,-1)$ and $i_s=t$ for some $s\in [1,k]$. Suppose $\lambda^r$ is the first generalized partition with $r>s$ and $\lambda^r_t=\lambda^r_{t-1}$. Then $i_s=\cdots=i_{r-1}=t$ and $i_{r}<t$.
\end{enumerate}
\end{lemma}

\begin{proof}
Assume $\lambda^1=(1,0,\ldots,0)$. . From Lemma~\ref{lem:labelsincrease}, we know that $i_r\geq t$. If $i_r=t$, then $\mu^{r-1}$ is obtained from $\lambda^r$ by removing a box in row $t$ and is thus not a generalized partition shape. We have that $i_s=\cdots=i_{r-1}$ from Lemma~\ref{lem:jumpgivesequality}. The argument for $\lambda^1=(0,\ldots,0,-1)$ is analogous.
\end{proof}

The following proposition follows from comparing the generalized oscillating promotion growth rules to the classical promotion growth rules.

\begin{proposition} Generalized oscillating promotion restricts to classical promotion when applied to a standard Young tableau $T$.
\end{proposition}

In the next section, we give a tableau interpretation of generalized oscillating promotion. We can construct the tableau promotion rules from the growth rules using the correspondence discussed in Section~\ref{sec:tabandprom}. Namely, reading each sequence $(\mu^0,\ldots,\mu^{s-1},\lambda^s,\ldots,\lambda^k)$, filling the box of $\mu^{j}/\mu^{j-1}$ with $j+1$ if $\mu^{j-1}\subseteq\mu^{j}$ or the box of $\mu^{j-1}/\mu^j$ with $j+1'$ if $\mu^{j}\subseteq \mu^{j-1}$; filling the box $\lambda^j/\lambda^{j-1}$ with $j$ if $\lambda^{j-1}\subseteq\lambda^j$ or the box of $\lambda^{j-1}/\lambda^{j}$ with $j'$ if $\lambda^j\subseteq \lambda^{j-1}$; and filling $\lambda^s/\mu^{s-1}$ with $\bullet$ if $\mu^{s-1}\subseteq\lambda^s$ or $\mu^{s-1}/\lambda^s$ with $\bullet'$ if $\lambda^s\subseteq \mu^{s-1}$ will construct the tableau associated to each step in the promotion. For example, reading $(\mu^0,\mu^1,\mu^2,\mu^3,\lambda^4,\lambda^5,\lambda^6,\lambda^7,\lambda^8)$ in Example~\ref{ex:gopgrowth} tells us that we have the following tableau as a step in the generalized oscillating promotion.
\begin{center}
\ytableausetup{boxsize=.25in}
\begin{ytableau}
\none & 28' & \bullet 5' \\
\none & 46'\\
*(red) 3'7
\end{ytableau}
\end{center}

\subsection{Generalized oscillating promotion on tableaux}
We now describe the same generalized oscillating promotion action $p: \text{GOT}(k,n)\to\text{GOT}(k,n)$ in terms of tableaux. 

\begin{definition}[Generalized Oscillating Promotion]\label{def:goptab}
Given a generalized oscillating tableau $T=(\emptyset,\lambda^1,\ldots,\lambda^k)\in\text{GOT}(k,n)$, form $p(T)$ using the following steps.
\begin{enumerate}
\item If the entry 1 exists, delete it and replace it with $\bullet$. If instead the entry $1'$ exists, delete it and replace it with $\bullet'$. At each step in the promotion, we denote the box currently containing $\bullet$ or $\bullet'$ by $b$, the row containing $b$ by $r_b$, and the column containing $b$ by $c_b$. Let $c_b+1$ denote the column to the right of $c_b$ and $c_b-1$ denote the column to the left of $c_b$.
\item Perform jeu de taquin: For each $i\in\{2,\ldots,n\}$, perform the following swap with $i$ starting at 2 and consecutively increasing after each swap.
\begin{enumerate}
\item If $\bullet$ is unprimed:
\begin{enumerate}
\item If the box containing $i$ is directly below or directly right of $b$, switch the labels within these two boxes.
\item If $\bullet$ is in the same box at $i'$ consider $\lambda^{i}$.
\begin{itemize}

\item If $\lambda^{i}_{r_b}\neq\lambda^{i}_{r_b+1}$ or $\lambda^{i}_{r_b+1}$ does not exist, 
delete $\bullet$ and $i'$ from $b$ and add the subset $\{i', \bullet\}$ to the box directly to its left.
\item If $\lambda^{i}_{r_b}=\lambda^{i}_{r_b+1}$, delete $\bullet$ and $i'$ from box $b$ and add the subset $\{i', \bullet\}$ to the box in column $c_b-1$ and row $r$, where $r$ is the bottommost row of $\lambda^i$ of size $\lambda^i_{r_b}$. 

\end{itemize}
\item If neither of the previous two things is true, do nothing.
\end{enumerate}
\item If $\bullet$ is primed:
\begin{enumerate}
\item If the box containing $i'$ is directly above or directly left of $b$, switch these labels within the two boxes.
\item If $\bullet'$ is in the same box at $i$, consider $\lambda^i$.
\begin{itemize}
\item[-] If $\lambda^{i}_{r_b}\neq\lambda^{i}_{r_b-1}$ or $\lambda^{i}_{r_b-1}$ does not exist, delete $\bullet'$ and $i$ from $b$ and add the subset $\{i, \bullet'\}$ to the box directly to its right.

\item[-] If $\lambda^{i}_{r_b}=\lambda^{i}_{r_b-1}$, delete $\bullet'$ and $i$ from box $b$ and add the subset $\{i, \bullet'\}$ to the box in column $c_b+1$ and row $r$, where $r$ is the topmost row of $\lambda^i$ of size $\lambda^i_{r_b}$.
\end{itemize}
\item If neither of the previous two things is true, do nothing.
\end{enumerate}
\end{enumerate}
\item Delete the $\bullet$ or $\bullet'$ and fill its box with $k+1$ or $k+1'$, respectively.
\item Subtract 1 from each entry.
\end{enumerate}
\end{definition}

See Figure~\ref{fig:genoscprom} for an example of generalized oscillating promotion using these tableau rules. Notice that in the example shown, $T$ and $p(T)$ do not have the same shape. 


\begin{proposition}\label{prop:growthandtableau}
The growth rules in Section~\ref{sec:genoscprogrowth} describe the generalized oscillating promotion in Definition~\ref{def:goptab}.
\end{proposition}
\begin{proof}
We use the fact that we can recover the generalized oscillating tableau with $\bullet$ or $\bullet'$ at each step in the promotion by reading the sequences $(\mu^0,\ldots,\mu^{s-1},\lambda^s,\ldots,\lambda^k)$ in the growth diagram. 

Rule (1) is clear from reading $(\mu^0,\lambda^1,\ldots,\lambda^k)$. We can recover each of the remaining tableau promotion rules from the generalized oscillating tableaux growth rules. We explain this is detail for growth rule (OP1a) when $j=i$ as the remaining arguments are similar.

Growth rule (OP1a) with $j=i$ gives tableau rule (2ai) in the case that $i$ is directly right of $\bullet$. To see this, note that (OP1a) says that $\lambda^s/\mu^{s-1}$ is one box in row $i$ and $\lambda^{s+1}/\lambda^s$ is one box in row $i$. Using our translation from growth diagrams to steps in the tableau promotion, this means that before swapping $\bullet$ with $s+1$, $\bullet$ is directly left of $s+1$ in row $i$. We can read this from the sequence $(\mu^0,\ldots,\mu^{s-1},\lambda^s,\ldots,\lambda^k)$. To see what happens after swapping $s+1$ and $\bullet$, we read $(\mu^0,\ldots,\mu^s,\lambda^{s+1},\ldots,\lambda^k)$. Growth rule (OP1a) says that $\mu^s/\mu^{s-1}$ is one box in row $i$ and $\lambda^{s+1}/\mu^s$ is one box in row $i$, and therefore $s+1$ is now directly left of $\bullet$ in this step of the tableau promotion. 

Using similar reasoning, we see that 

\begin{itemize}
\item[-] growth rule (OP1b) gives (2ai) when $i$ is directly below $\bullet$,
\item[-] growth rules (OP1a) and (OP1c) together give tableau promotion rule (2aiii),
\item[-] growth rule (OP1c) with $j=i$ gives (2aii) when $\lambda^i_{r_b}\neq\lambda^i_{r_b+1}$, and
\item[-] growth rule (OP1d) gives (2aii) when $\lambda^i_{r_b}=\lambda^i_{r_b+1}$.
\end{itemize} 

We can analogously pair growth rules under (OP2) with tableau promotion rules within (2b). Rules (3) and (4) follow from reading the resulting $(\mu^0,\ldots,\mu^k)$ in the growth diagram. 
\end{proof}

\subsection{Promotion on generalized oscillating tableaux coming from webs}

We restrict our attention to generalized oscillating promotion on generalized oscillating tableaux coming from webs, i.e., generalized oscillating tableaux of the form $T=(\lambda^0,\ldots,\lambda^k)\in\text{GOT}(k,3)$ with $\lambda^k=(m,m,m)$ for some $m\in\mathbb{Z}$. In particular, we first consider generalized oscillating tableaux that come from webs that do not contain the identity web.

\begin{proposition}\label{prop:isvalues}
\begin{enumerate}
\item Suppose $T=(\lambda^0,\ldots,\lambda^k)$ comes from a  web that does not contain the identity web and $\lambda^1=(1,0,0)$. Let $i_s$ denote the row of the box added to $\mu^{s-1}$ to obtain $\lambda^s$. Then there exist $m<n\in[1,k]$ such that
\begin{enumerate}
\item $i_s=1$ for $s\leq m$,
\item $i_s=2$ for $m<s<n$, and 
\item $i_s=3$ for $s \geq n$.
\end{enumerate}

\item Similarly, if $\lambda^1=(0,0,-1)$ and $i_s$ denotes the row of the box deleted from $\mu^{s-1}$ to obtain $\lambda^s$, then there exist $m<n\in[1,k]$ such that
\begin{enumerate}
\item $i_s=3$ for $s\leq m$,
\item $i_s=2$ for $m<s<n$, and 
\item $i_s=1$ for $s \geq n$.
\end{enumerate}
\end{enumerate}
\end{proposition}
\begin{proof}
Suppose $\lambda^1=(1,0,0)$ and recall from Lemma~\ref{lem:labelsincrease} that $1=i_1\leq\cdots\leq i_k$ From Proposition~\ref{prop:hitextremerays}, there is some first $m\geq 1$ such that $\lambda^{m+1}_1=\lambda^{m+1}_2$ and $\lambda^{m+1}_2\neq\lambda^{m+1}_3.$  It then follows from Lemma~\ref{lem:equalitygivesjumps} that $i_1=\cdots=i_m=1<i_{m+1}$. We have thus applied with (OP1b) or (OP1d) to $\mu^{m-1}$. 

If we applied (OP1b), then $i_{m+1}=2$ from the growth rule because we must have had $j=2$. If we applied (OP1d), then we know that $\mu^{m-1}=\lambda^{m+1}$. Hence $\mu^{m-1}_2\neq\mu^{m-1}_3$, and so we add a box to $\mu^{m-1}_2$ to obtain $\mu^m$ and then subtract a box in the second row to obtain $\lambda^{m+1}$ from $\mu^m$. Therefore $i_{m+1}=2$.

Since $\lambda^k_1=\lambda^k_2=\lambda^k_3$, there is some first $\lambda^{n}$ with $m<n\leq k$ and $\lambda^n_2=\lambda^n_3$. Hence by Lemma~\ref{lem:equalitygivesjumps}, we apply either (OP1b) or (OP1d) to $\mu^{n-2}$ and see that $i_{n-1}=2$ and $i_n=3$.
\end{proof}

\begin{corollary}\label{cor:tabrowchange}
Suppose generalized oscillating tableau $T=(\lambda^0,\ldots,\lambda^k)$ comes from a web $D$ with leftmost vertex $v$ and $D$ does not contain the identity web. The entry $i$ (or $i'$) changes rows in step (2) of generalized oscillating promotion if and only if

\begin{enumerate}
\item $v$ is black and $\lambda^i$ is either
\begin{enumerate}
\item the first generalized partition of $T$ with $\lambda^i_1=\lambda^i_2$ or
\item the first generalized partition of $T$ with $\lambda^i_2=\lambda^i_3$ to the right of the first generalized partition of $T$ with first row equal in size to second row.
\end{enumerate}
\item $v$ is white and $\lambda^i$ is either
\begin{enumerate}
\item the first generalized partition of $T$ with $\lambda^i_2=\lambda^i_3$ or
\item the first generalized partition of $T$ with $\lambda^i_1=\lambda^i_2$ to the right of the first generalized partition of $T$ with second row equal in size to first row.
\end{enumerate}
\end{enumerate}
In addition, when an entry changes row, it goes down one row if unprimed and up one row when primed.
\end{corollary}
\begin{proof}
Assume $v$ is black, and let $i_s$ denote the row in which a box is added to $\mu^{s-1}$ to obtain $\lambda^s$ in the generalized oscillating promotion growth rules. Then Lemma~\ref{lem:equalitygivesjumps} and Lemma~\ref{lem:jumpgivesequality}
 together imply that $i_s<i_{s+1}$ if and only if $\lambda^{s+1}$ is the first partition after $\lambda^s$ with $\lambda^{s+1}_{i_s}=\lambda^{s+1}_{i_s+1}.$ Proposition~\ref{prop:isvalues} says that when this happens, $i_{s+1}=i_s+1$. Proposition~\ref{prop:growthandtableau} gives the equality of the growth rules and the tableau rules for promotion. By using the translation between the growth rules and the tableau rules, we see that $i_s<i_{s+1}$ means that $s+1$ has changed row. The analogous argument works for $v$ white.
\end{proof}

\section{Rotation corresponds to generalized oscillating promotion}\label{sec:rotationandgenoscprom}

We heavily use the results of T.K. Peterson, P. Pylyavskyy, and B. Rhoades \cite{petersen2009promotion}.

Suppose we have a web $D$ with chosen leftmost vertex $v$. Without loss of generality, suppose $v$ is black as analogous arguments always hold for $v$ white. We extend $D$ to a larger web $D'$ with all black boundary vertices by replacing each white boundary vertex of $D$ by a fork with two new black boundary vertices as shown below. We refer to the shape on the right of a \textit{fork} and we say that the two new black vertices are ``in the same fork.''

\begin{center}
\raisebox{.2in}{\begin{tikzpicture}
\node (2) at (0,-1) {};
\draw (0,-1)--(0,0);
\node[circle,draw=black, fill=white, inner sep=0pt,minimum size=4pt] (1) at (0,0) {};
\end{tikzpicture}}\hspace{.3in}\raisebox{.4in}{$\longrightarrow$}\hspace{.3in}
\begin{tikzpicture}
\node (-1) at (1,1) {$\bullet$};
\node (0) at (-1,1) {$\bullet$};
\node (2) at (0,-1) {};
\draw (0,0)--(0,-1);
\draw (-1,1)--(0,0);
\draw (1,1)--(0,0);
\node[circle,draw=black, fill=white, inner sep=0pt,minimum size=4pt] (1) at (0,0) {};
\end{tikzpicture}
\end{center}

We may now apply the results of Peterson--Pylyavskyy--Rhoades. Namely, we may use the same chosen leftmost vertex $v$ for $D'$ and compute the state string for $D'$, $w(D')$.  Then we know exactly how to compute the word obtained after rotation one vertex counterclockwise, $w(p(D'))$ from Theorem~\ref{thm:allblackwordchange}. We use this to obtain the results in this section.

Recall that $v^l$ denotes the vertex at the end of the left cut starting at $v$ and $v^r$ denotes the vertex and the end of the right cut starting at $v$. 

\begin{lemma}
Let $D$ be a web that does not contain the identity web and has leftmost vertex $v$, and let $p(D)$ denote the result of rotating one vertex counterclockwise. The vertices that have different states in $D$ and $p(D)$ are exactly $v$, $v^l$, and $v^r$. 
\end{lemma}

\begin{proof}
Extend the web $D$ to $D'\in\mathcal{M}_n$ as described. Web $D'$ has leftmost vertex $v$, and the left and right cuts end at some $(v^l)'$ and $(v^r)'$, which may or may not be distinct from $v^l$ and $v^r$. If $(v^l)'=v^l$, it is clear that the state of $v^l$ must change. If $(v^l)'\neq v^l$, then $v^l$ is the penultimate vertex in the left cut starting at $v$. Let $b$ denote the vertex in the same fork as $(v^l)'$. Since the left and right cuts to not intersect and $D$ does not contain the identity web, we know that the state of $b$ is the same in $D'$ as in $p(D')$. Thus, changing the state of $(v^l)'$ must also change the state of $v^l$. Similarly for $v^r$.

The fact that the state of $v$ must change follows from the observation that any dominant signature and state string must begin with either $(1,\bullet)$ or $(\bar{1},\circ)$ and end with either $(\bar{1},\bullet)$ or $(1,\circ)$.

Since no other states of $D'$ change, no other states of $D$ change.
\end{proof}

Given a web $D$ with leftmost vertex $v$, its signature as well as the signature for any rotation are clearly determined. We may thus suppress this from the signature and state string and consider only the state string. As in the case when the boundary vertices are all the same color, we will think of the corresponding state string as a word $w(D)$ and refer to this as the \textit{word corresponding to $D$ with leftmost vertex $v$.}

\begin{theorem}\label{thm:generalwordchange}
Let $D$ be a web with leftmost vertex $v$ that does not contain the identity web. We denote the corresponding word $w=w_1\cdots w_n$. Suppose $v^l$ corresponds to $w_a$ and $v^r$ to $w_b$. Then the word obtained after rotating $D$ one vertex counterclockwise is 

\[w'=w_2\cdots w_{a-1}w_a'w_{a+1}\cdots w_{b-1}w_b' w_{b+1}\cdots w_n w_1',\]
where 
\begin{itemize}
\item $w_1'=1$ if $v$ is white and $\bar{1}$ if $v$ is black,
\item $w_a'=1$ if $v^l$ and $v$ are black, $w_a'=\bar{1}$ if $v^l$ and $v$ are white, and $w_a'=0$ if $v$ and $v^l$ are different colors.
\item $w_b'=0$ if $v^r$ and $v$ are the same color, $w_b'=\bar{1}$ if $v$ is black and $v^r$ is white, and $w_b'=1$ if $v$ is white and $v^r$ is black.
\end{itemize}
\end{theorem}
\begin{proof}
The result for $w_1$ follows from the fact that the signature and state string is dominant.

We now explain the result for $w_a'$ when $v$ is black. If $v^l$ is black, then $v^l=(v^l)'$, and the fact that $w_a'=1$ follows from Theorem~\ref{thm:allblackwordchange}. If $v^l$ is white, then clearly $(v^l)'\neq v^l$. In particular, we are in the situation below.
\begin{center}
\begin{tikzpicture}
\node (a) at (0,0) {$\bullet$};
\node (b) at (2,0) {$\bullet$};
\node (c) at (1.3,-1) {$v^l$};
\draw[-] (1,-1) -- (0,0)  node [above, label=right: $(v^l)'$] {};
\draw[-] (1,-1) -- (1,-2);
\draw[-] (1,-1) -- (2,0) node [above, label=left:$b$] {};
\node[circle,draw=black, fill=white, inner sep=0pt,minimum size=4pt] (d) at (1,-1) {};
\end{tikzpicture}
\end{center}
We know that $(v^l)'$ is the left vertex of the fork since the last turn of a left cut that starts and ends at a black vertex is a left turn. From Theorem~\ref{thm:allblackwordchange} and the fact that left and right cuts do not intersect, upon rotation, the state of $(v^l)'$ changes from 0 to 1 and the state of $b$ remains the same. Using Remark~\ref{rem:tricolor}, this implies that the state of $b$ must be $\bar{1}$, and so $w_a=1$ and $w_a'=0$, as desired. 

The arguments for $w_a'$ when $v$ is white and for $w_b'$ are similar.
\end{proof}

\subsection{Rotations and generalized oscillating promotion}
We can now relate this to generalized oscillating promotion.

\begin{proposition}\label{prop:rotationextremerays}
Let $D$ be a web that does not contain the identity web with chosen leftmost vertex $v$. Let $v^l$ and $v^r$ be vertices in $D$ defined as before.
\begin{enumerate}
\item The dominant path associated to the states of the vertices preceding $v^l$ inclusively ends at the upper extreme ray if $v$ is black and ends at the lower extreme ray if $v$ is white. Further, $v^l$ is the leftmost vertex with this property.

\item The dominant path associated to the states of the vertices preceding $v^r$ inclusively ends at the lower extreme ray if $v$ is black and ends at the upper extreme ray if $v$ is white. Further, $v^r$ is the leftmost vertex to the right of $v^l$ with this property.
\end{enumerate}
\end{proposition}

\begin{proof}
Suppose $v$ is black. Let $D'\in\mathcal{M}_n$ denote the extension of $D$, and consider how the dominant path for $D$ compares to that for $D'$: for every white boundary vertex of $D$, we replace the corresponding step in the dominant path for $D$, $\mu_k$, by two steps $\mu_k'$ and $\mu_k''$, where $\mu_k=\mu_k'+\mu_k''$. We can reverse this procedure to build the path for $D$ from the path for $D'$. For example, if we had white vertex with state $(0,\circ)$ in $D$, then we replace the step in the dominant path for $D$ corresponding to $(0,\circ)$ with the step corresponding to $(1,\bullet)$ plus the step corresponding to $(\bar{1},\bullet)$ to build the path for $D'$. Since the original segment and its replacement have the same starting and ending points, any point where the path for $D$ touches an extreme ray is also a point where $D'$ touches an extreme ray.

If $v^l$ is black, the corresponding result follows from Lemma~\ref{lem:balanced}.

If $v^l$ is white, then $v^l$ is extended by $(v^l)'$ and $b$ in $D'$ as in the proof of Theorem~\ref{thm:generalwordchange}, and the step in the path corresponding to $(v^l)'$ is the first place the path for $D'$ returns to the upper extreme ray. To obtain the path for $D$, we replace the step in direction $(0,\bullet)$ coming from $(v^l)'$ followed by the step in direction $(-1,\bullet)$ coming from $b$ by one step in direction $(1,\circ)$, corresponding to $v^l$. We see that this modified path still returns to the upper extreme ray after the step corresponding to $v^l$. 

We can use a similar argument for the remaining cases.
\end{proof}

We now prove Theorem~\ref{thm:mainthm}.

\begin{theorem*}
For any web $D$ with chosen leftmost vertex $v$, we have 
\[T(w(p(D))=p(T(w(D))).\] That is, the generalized oscillating tableau associated with the rotation of $D$ is given by generalized oscillating promotion of the tableau associated with $D$ itself.
\end{theorem*}
\begin{proof}
Assume that $D$ does not contain the identity web or $v$ is not a vertex in the identity web. Then Proposition~\ref{prop:rotationextremerays}, Theorem~\ref{thm:generalwordchange}, and Corollary~\ref{cor:tabrowchange} along with the observation that parts (1c) and (2c) of Proposition~\ref{prop:isvalues} imply the state change for $v$ give the desired result.

Now suppose $v$ is black and is a vertex of the identity web. Let $((j_1,s_1),\ldots,(j_k,s_k))$ denote the signature and state string for $D$, and suppose the other vertex in the identity web is the $m$th vertex starting from $v$ and counting clockwise. If $m=2$, it is easy to see that $i_1=i_2=1$ and $i_3=3$, where $i_s$ is the row in which one adds a box to $\mu^{s-1}$ to obtain $\lambda^s$. This proves the result in this case. And it is impossible for $m=3$ since the boundary vertex to the right of $v$ must have exactly one edge adjacent to it. We may thus assume that $m\geq 4$. 

Using the argument from the proof of Proposition~\ref{prop:hitextremerays}, the segment $((j_2,s_2),\ldots,(j_{m-1},s_{m-1}))$ is also a dominant signature and state string. Thus none of the points in the dominant path for $D$ $\pi_0,\pi_1,\ldots,\pi_m$ touch the upper extreme ray and not the lower extreme ray of the dominant chamber, and the first point on the upper extreme ray, $\pi_m$ is also on the lower extreme ray. 

Let $T(D)=(\lambda^0,\ldots,\lambda^k)$ be the generalized oscillating tableau associated to $D$. It follows that the first $\lambda^j$ with $\lambda^j_1=\lambda^j_2$ also has $\lambda^j_2=\lambda^j_3$, and this happens at $\lambda^m$. Thus $1=i_1=\ldots=i_{m-1}$ and $i_m=3$. We conclude that only $m$ changes row in step (2) of promotion, and it goes from row 1 to row 3. In addition, $i_k=3$ corresponds to the state for $v$ changing from 1 to $\bar{1}$, as desired.

\end{proof}

\section{Future directions}\label{sec:future}
\subsection{Enumeration and cyclic sieving}
Let $X$ be a finite set, $C=\langle c \rangle$ be a finite cyclic group acting on $X$, and $X(q)\in\mathbb{Z}[q]$ be a polynomial with integer coefficients. Then the triple $(X,C,X(q))$ \textit{exhibits the cyclic sieving phenomenon} \cite{reiner2004cyclic} if for each $d>0$, $|X^{c^d}|=X(\zeta^d)$, where $\zeta\in\mathbb{C}$ is a $|C|$th root of unity and $X^{c^d}$ is the set of fixed points of the action of $c^d$.

In \cite{rhoades2010cyclic}, B. Rhoades shows that standard Young tableau promotion on rectangular tableaux exhibits the cyclic sieving phenomenon. In \cite{petersen2009promotion}, the authors reprove this result in the special case that the tableaux have two or three rows using the connection between promotion and webs. 

\begin{theorem}[\cite{petersen2009promotion,rhoades2010cyclic}]
Let $\lambda\vdash N=bn$ be a rectangle with $b=2$ or 3 rows and let $C=\mathbb{Z}/n\mathbb{Z}$ act on $X=\{\text{standard Young tableaux of shape }\lambda\}$ by promotion. Then the triple $(X,C,X(q))$ exhibits the cyclic sieving phenomenon, where 
\[X(q)=\frac{[n]_q!}{\prod_{(i,j)\in\lambda} [h_{ij}]_q}\] is the $q$-analogue of the hook length formula. 
\end{theorem}

We would be interested in knowing if it is possible to extend this result to the generalized oscillating tableaux corresponding to webs.

\subsection{sl$(n)$ webs}
The webs described here correspond to sl$(3)$, and webs corresponding to sl$(n)$ for $n>3$ are much less developed. See, for example, \cite{cautis2014webs,fraser2017dimers,kim2003graphical}. In particular, when $n>3$ there is no appropriate notion of an irreducible web and no rotation-invariant basis of webs.

However, given a definition of a signature and state string for sl$(n)$ webs, we think it is possible that our generalized oscillating promotion describes rotation of these webs. Specifically, perhaps the promotion $p:\text{GOT}(k,n)\to\text{GOT}(k,n)$ describes rotation for webs corresponding to sl$(n)$. We state these ideas as conjectures. 

\begin{conjecture}\label{conj:bijection} There is a bijection between generalized oscillating tableaux of length $k$ with $n$ parts such that the last component is $(m,\ldots,m)$ for some $m\in\mathbb{Z}$ and sl$(n)$ webs with $k$ boundary vertices.
\end{conjecture}

Assuming Conjecture~\ref{conj:bijection} holds, we also have the following conjecture. Suppose $D$ is an sl$(n)$ web with chosen leftmost vertex. As before, let $T(D)$ denote the generalized oscillating tableau (conjecturally) associated to $D$, $p(D)$ denote the result of rotating $D$ one vertex counterclockwise, and $w(D)$ denote the word obtained from the states corresponding to the boundary vertices of $D$.

\begin{conjecture}
For any sl$(n)$ web $D$ with chosen leftmost vertex, we have 
\[T(w(p(D))=p(T(w(D))).\] That is, the generalized oscillating tableau associated with the rotation of $D$ is given by generalized oscillating promotion of the tableau associated with $D$ itself.
\end{conjecture}

\section*{Acknowledgements}
The author thanks Hugh Thomas for suggesting this problem and for his support; Chris Fraser, Alexander Garver, and Pasha Pylyavskyy for helpful discussions; and V\'eronique Bazier-Matte, Guillaume Douville, Alexander Garver, Hugh Thomas, and Emine Yildirim for their collaboration in the LaCIM working group 2016-2017, in which webs were the topic. The author received support from NSERC, CRM-ISM, and the Canada Research Chairs Program.

\bibliographystyle{plain}
\bibliography{main.bib}

\end{document}